\documentclass[a4paper, reqno, 11pt]{article}

\usepackage{mathtools,amssymb,amsthm,mathrsfs,calc,graphicx,xcolor,cleveref,dsfont,tikz,pgfplots,bm, extarrows}
\usepackage[left=3.45cm, top=2.45cm,bottom=2.65cm,right=3.45cm]{geometry}
\usepackage{bbm}
\usepackage{enumitem} 
\usepackage[numbers]{natbib}

\usepackage{float}
\usepackage{mathtools}

\usepackage[british]{babel}

\renewcommand{\deg}{\texttt{deg}}
\newlength{\NOTskip}

\newtheorem{theorem}{Theorem}[section]
\newtheorem{lemma}[theorem]{Lemma}
\newtheorem{prop}[theorem] {Proposition}

\newtheorem{definition}[theorem] {Definition}
\newtheorem{corollary}[theorem] {Corollary}

\theoremstyle{definition}

\newtheorem{remark}[theorem] {Remark}

\newcommand{\dint}{{\operatorname {d}}}

\def\phi{\varphi }

\newcommand{\R}     {\mathbb{R}}

\newcommand{\N}     {\mathbb{N}}
\renewcommand{\P}   {\mathbb{P}}
\newcommand{\E}     {\mathbb{E}}

\renewcommand{\d}{\, d}

\renewcommand{\d}   {\texttt{D}\!}
\newcommand{\Cov}   {\operatorname{Cov}}

\newcommand{\1}{{\bf 1}}

\newcommand{\cG}{\mathcal{G}}
\newcommand{\cGi}{\mathcal{G}^{(i)}}

\newcommand{\out}{{\textnormal{\deg}^{+}}}

\newcommand{\ind}{\textnormal{\deg}^-}

\newcommand{\indi}{\textnormal{\deg}^{(i,-)}}
\newcommand{\Var}{\operatorname{Var}}
\renewcommand{\d}{\textnormal{\texttt{d}}^-}
\renewcommand{\o}{\textnormal{\texttt{d}}^+}
\newcommand{\di}{\textnormal{\texttt{d}}^{(-,i)}}

\newcommand{\oi}{\textnormal{\texttt{d}}^{(+,i)}}

\newcommand{\D}{\textnormal{\texttt{D}}}
\newcommand{\cd}{\textnormal{\texttt{d}}^{-}}
\newcommand{\cdi}{\bar{\textnormal{\texttt{d}}}^{(-,i)}}
\newcommand{\coi}{\bar{\textnormal{\texttt{d}}}^{(+,i)}}

\makeatletter
\def\moverlay{\mathpalette\mov@rlay}
\def\mov@rlay#1#2{\leavevmode\vtop{%
		\baselineskip\z@skip \lineskiplimit-\maxdimen
		\ialign{\hfil$\m@th#1##$\hfil\cr#2\crcr}}}
\newcommand{\charfusion}[3][\mathord]{
	#1{\ifx#1\mathop\vphantom{#2}\fi
		\mathpalette\mov@rlay{#2\cr#3}
	}
	\ifx#1\mathop\expandafter\displaylimits\fi}
\makeatother

\allowdisplaybreaks[1] 
\begin{document}

\title{ \bfseries A central limit theorem for the number of isolated vertices in a preferential attachment random graph \\
}
\renewcommand{\thefootnote}{\fnsymbol{footnote}}
\author{\textsc{Carina Betken\footnotemark[1]}}

\renewcommand{\thefootnote}{\fnsymbol{footnote}}
\footnotetext[1]{Ruhr-Universität Bochum, Germany. Email: \texttt{carina.betken@rub.de}}
 
\maketitle

\begin{abstract}
  \noindent We study the number of isolated vertices in a preferential attachment random graph introduced by Dereich and Mörters in 2009. 
In this graph model vertices are added over time and newly arriving vertices connect to older ones with probability proportional to a (sub-)linear function of the indegree of the older vertex at that time. 
 Using Stein's method and a size-bias coupling, we deduce bounds in the Wasserstein distance between the law of the properly rescaled number of isolated vertices and a standard Gaussian distribution. 

  \par\medskip
\footnotesize
\noindent{\emph{2010 Mathematics Subject Classification}}: Primary 05C80, Secondary 60F05

  \par\medskip
\noindent{\emph{Keywords:} random graphs; preferential attachment; Stein's method; size-bias coupling; rates of convergence}
\end{abstract}

\tableofcontents

\section{Introduction}

Many structures  in science and nature in which components interact with one another can be modelled and analysed with the help of random networks. Each component is typically represented by a node and relations between components are indicated by edges connecting the corresponding vertices. There are numerous examples of structures that can be modelled in such a way including for example molecules in metabolisms, agents in technological systems and people in social networks, to name just a few. For more details and an overview over the mathematical research field of random networks we refer the reader to \cite{Hofstad2017}.\\

In order to better understand the structure of random graphs, the study of degree distributions and subgraph count statistics has been an active field of research since the introduction of the first mathematically rigorous random graph model by Erdös and Rényi in \cite{ER59} in the late 1950s. A random graph in this model consists of a fixed number $ n $ of vertices and a random number of edges, where each edge exists independent of all others with some fixed probability $ p $. 
The number of small subgraphs, and triangles in particular,  in this  graph model was studied in  \cite{Rucinski88} and \cite{Roellin2017}. While the first of these uses cumulant bounds to show asymptotic normality, the latter makes use of a variation of Stein's method, the so-called Stein-Tikhomirov method that combines Stein's method with characteristic functions.  In \cite{Krokowski2017} and \cite{Barbour89} the authors study the number of vertices with a prescribed degree as well as subgraph count statistics.
In both works Stein's method is used to show convergence towards a Gaussian limit.
 \\
As mentioned above, one statistic which has been a frequent object of study is the number of vertices being directly connected to a fixed number $ d $ of other vertices. In \cite{Goldstein2013} the author derives a new  Berry–Esseen bound for sums of dependent random variables combining Stein’s method, size-bias couplings and an inductive technique, 
   to assess  the accuracy of the normal approximation for  the distribution of the number of vertices of a given degree  in the classical Erd\"os-Rényi  random graph with parameter $ p\approx \frac{\theta}{n-1}, ~\theta>0 $. This generalizes the result obtained by Kordecki handling the special case $ d=0 $,  see \cite{Kordecki90}. 
 More recently, in \cite{BRR2019} Barbour, R\"ollin and Ross used Stein couplings to deduce optimal bounds between the number of isolated vertices in the Erd\"os-Rényi random graph with parameter $ p\approx \lambda/n $ and the truncated Poisson distribution, strengthening the results given in \cite{RR2015}. Stein's method was also employed in \cite{Fang2014} to derive error  bounds  in total variation distance between a discretized normal  distribution and the number of vertices with a given degree in the Erd\"os-Rényi random graph and the uniform multinomial occupancy model. 
 The inhomogeneous random graph model 
  was dealt with in \cite{Penrose2018}. Using Stein's method, the author could show that in this model the number of isolated vertices  asymptotically follows a Poisson distribution. \\
\\
Due to its staightforward construction rules, which account for a lot of independencies, random quantities in the Erd\"os-Rényi random graph can often be considered in applications of rather general results, see for instance \cite{Goldstein2013} and \cite{Krokowski2017}.
However, this graph model does not explain the specific structures observed in many real world networks such as the World Wide Web, social interaction or biological neural networks, which usually exhibit powerlaw degree distributions. The principle of preferential attachment has become a well-known concept to explain the occurrence of these kinds of structures. Preferential attachment networks typically include two characteristic features: they are dynamic in the sense that vertices are successively added over time  and new vertices prefer to connect to older vertices, which are already well connected in the existing network.
The construction rules for such networks can be made precise in various ways, so that starting with the pioneering work \cite{Barabasi1999} of Barab\'asi and Albert, various different models of preferential attachment random graphs have appeared in the scientific literature in recent years (see for example \cite{Barabasi1999},  \cite{KR01}, \cite{Oliveira2005}, \cite{RTV07}, \cite{Dereich2009},  \cite{Ross2013}, \cite{JM2015} and \cite{GGLM2019})).  
Dependency structures in preferential attachment random graphs are clearly  more complex than those seen in the Erd\"os-Rényi random graph. Hence, results are in general less numerous and usually heavily dependent on the model at hand.  
In 
 \cite{PRR13} Pek\"oz, R\"ollin and Ross  successfully  applied Stein's method  to prove a rate of convergence in the  Kolmogorov distance for the indegree distribution of any fixed vertex to a power law distribution by comparing it to a mixed negative binomial distribution, whereas in  \cite{PRR17} the same authors  prove rates of convergence in the multidimensional case for joint degree distributions. One feature inherent to these models as well as to the Barab\'asi-Albert model is that every vertex connects to a fixed number of vertices when entering the network. In contrast, the model introduced in  \cite{Dereich2009} allows for random outdegrees, which seems to be a reasonable assumption. In the same work, Dereich and Mörters deduce the asymptotic indegree distribution in that model to be of the form
 \[
 \mu(k)=\frac{1}{1+f(k)}\prod_{i=1 }^{k-1}\frac{f(i)}{f(i)+1},
 \]
 where $ f $ denotes the so-called attachment function (see Section~\ref{sec:model} for details).
 Depending on this function, $ \mu $ can be a power-law or an exponentially decaying distribution. Developing Stein's method for this class of limiting distributions, the authors in  \cite{BDO18} give error bounds in the total variation distance between the indegree distribution and the corresponding limit 
   for that very same model. The same work also provides rates of convergence for the outdegree distribution towards a Poisson limit.\\
    In \cite{Dereich2013} the authors look at component sizes in this model and give an abstract criterion for the existence of a giant component for general concave attachment functions $ f $,
   which becomes explicit when restricting to linear functions. 
   \\
    \\
    An important aspect of  the model described in \cite{Dereich2009} is that the outdegree of a vertex can be zero, so that vertices with neither incoming nor outgoing edges, might emerge. 
 In the present paper we study the distribution of the number of these isolated vertices. More precisely, using Stein's method we are able to derive a central limit theorem for the number of isolated vertices in the model introduced by Dereich and M\"orters. 
 We use a result given in \cite{GR96}, which provides a general bound on the proximity of a properly rescaled random variable to the standard Gaussian distribution with the help of a size-bias coupling. To apply it to our setting, we define a random graph
in  which the number of isolated vertices follows the size-bias version of the distribution of the number of  isolated vertices in the original graph. 
  We also obtain rates of convergence, which crucially depend on the maximal growth behaviour of the attachment function $ f $. 
 \\
 \\
 The rest of the paper is structured as follows. In Section~\ref{sec:backgroundmaterial} we introduce the preferential attachment model described in \cite{Dereich2009} and state established results that we will rely on in our proofs. In Section~\ref{chapter:main result} we formulate our main result, the central limit theorem for the rescaled number of isolated vertices. Section~\ref{sec:SBconstruction} gives the construction of a random graph in which the number of isolated vertices follows the size-bias distribution of the number of isolated vertices in the original graph. Section~\ref{sec:proofMainresult} contains the proof of Theorem~\ref{Thm:isolatedvertices}. The proofs of the auxiliary lemmas needed to prove Theorem~\ref{Thm:isolatedvertices} can be found in Section~\ref{sec:proofs}.

\section{Preliminaries}\label{sec:backgroundmaterial}
In this chapter we provide background material on the underlying random graph model and  the main methods of proof. Let us start with some notational clarifications: by  $ \mathbb{N}_0 $ we denote the set $ \mathbb{N} \cup \{0\} $ and we write $ [n]:=\{1,\ldots,n\} $. Furthermore, for  functions $ g,f :\mathbb{N}_0^j\rightarrow \R$ we write
$ g\lesssim f $ if there exists a constant $ C \in \R$ such that
\[
g(i_1, \ldots, i_j) \leq C \, f(i_1, \ldots, i_j)
\]
for all $(i_1, \ldots i_j) \in \mathbb{N}_{0}^j$.
Similarly, we write
 $ g \asymp f $ if $ g\lesssim f $ and $ f\lesssim g$.
 
  For functions $f: \N_0\rightarrow \R$, we define $ \Delta f(k):= f(k)-f(k-1) $ and put $ f(-1)=0 $.  
Throughout the paper we will frequently use the following integral test for convergence: for any decreasing function $ g:\R\rightarrow \R^+ $ we have 
\begin{align*}
 \int_{k}^{n}g(x)\, \dint x\ \leq\  \sum_{\ell=k}^n g(\ell) \ \leq \ \int_{k-1}^n g(x) \, \dint x,
\end{align*}
so that for any $ \alpha<0  $
\begin{align}\label{integraltest}
 \sum_{\ell=k}^n \ell^\alpha\ \asymp \ \begin{cases}
 n^{\alpha+1} &\mbox{ for } \alpha>-1,\\
 \log(n) &\mbox{ for } \alpha=-1,\\
 k^{\alpha+1}&\text{ for } \alpha<-1,
 \end{cases}
\end{align}
where the implicit constant might depend on $ \alpha $.

\subsection{Model}\label{sec:model}
As mentioned before, the model we study was introduced  in~\cite{Dereich2009} and can be described as follows: we start with a graph $\cG_1$ consisting of one vertex (labelled $ 1 $) and no edges. At each discrete time step $ n $ we add one vertex labelled $ n $ to the network, and independently for each $ k \in [n-1]$ we add a directed edge from $ n $ to $ k $ with probability
\begin{align}\label{connectionprob}
\frac{f(\deg^-_{n-1}(k))}{n-1},
\end{align}
where $\deg_{n-1}^-(k)$ denotes the \textbf{indegree} of vertex $k$, i.e. the number of edges $ (m,k) $ pointing from younger vertices $ m$ to the older  vertex $ k $ in the graph $\cG_{n-1}$ on $ n-1 $ vertices. Note that we say vertex $ m $ is younger than vertex $ k $ if $ m >k $. The \textbf{attachment function} $f : \N_0 \rightarrow (0,\infty)$ is assumed to satisfy $f(n) \leq n+1$, so that the expression in \eqref{connectionprob} in fact lies between zero and one.
An example of a preferential attachment graph on 35 vertices  build according to these rules is depicted in Figure~\ref{Gn}.

Note that the probability of vertex $ n $ connecting to some older vertex $ k $ is given by
\begin{align}\label{eq:ProbOfEdge}
\P((n,k)\in \cG_n)=\E \left[\E \left[\1\{ (n,k)\in \cG_n\}\vert \mathcal{G}_{n-1}\right]\right]=\frac{\E\left[f(\deg_{n-1}^{–}(k))\right]}{n-1},
\end{align}
where $ \{ (n,k)\in \cG_n\}\ $ denotes the event that there exists an edge between vertices $ n $ and $ k $ with $ n > k $.
\begin{figure}[t]
	\centering
	\includegraphics[width=0.85\columnwidth]{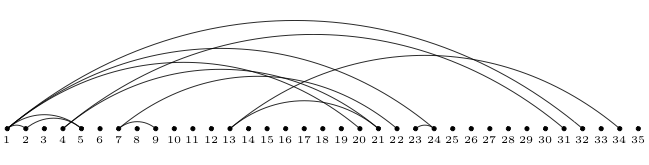}
	\caption{\small A possible realization of $ \cG_{35} $ with attachment function $ f(k)=\frac{1}{2}\sqrt{k}+ \frac{1}{2}. $}
	\label{Gn}
\end{figure}
Here, for fixed $ n $ the outgoing connections to older vertices are sampled independently, so that in contrast to many other models, as for instance those considered in \cite{Barabasi1999}, \cite{KR01} \cite{Oliveira2005}, \cite{Ross2013} and \cite{RTV07}, the \textbf{outdegree} $ \out (m) $, i.e. the number of edges $ (m,k) $ pointing from vertex $ m $ to older vertices $ k $, of every vertex $ m $ is random and can be zero.
 After $ n $ time steps, the graph $\cG_n$ consists of $n$ vertices and a random number of edges, where loops or multiple edges do not occur. 
 
 Note that the outdegree of every vertex is fixed after the time step in which it was inserted into the network. Also note, that although the existence of edges does depend on the existence of other edges, which makes the network in general more complicated to deal with than for example the Erdös-Rényi graph, the definition of the model, in particular the fact that decisions for outgoing edges of a fixed vertex are made independently from one another, brings about certain independence structures. As we will exploit these repeatedly throughout the proofs of our results, we will state them in a concise form here:
 \begin{itemize}
 	\item[\texttt{\textbf{(Ia)}}] For every $ m \in [n $] decisions for outgoing edges of vertex $ m $ are made independently from one another, i.e. for each $ k \in [m-1] $ the existence of edge $ (m,k) $ is independent of the existence of edge $(m, j) $ for all $ j \in [m-1], j \neq k $. In particular, this means that indegrees of distinct vertices are independent, i.e. $ \ind_m(k) $ and $ \ind_m(\ell) $ are independent random variables for $ k\neq \ell $ and all $ m \geq k,\ell $.
 	\item[\texttt{\textbf{(Ib)}}] For  each vertex $ k $ its indegree is independent of its outdegree, i.e. $ \out(k) $ and $ \ind_m(k) $ are independent random variables for all $ m \geq k $.
 	\item[\texttt{\textbf{(Ic)}}] For any $ m,k $ with $ m>k  $ the outdegree of $ k $ is independent of the indegree of vertex $ m $, i.e. for any  $ n,m,k $ such that $ n >m>k  $ the random variables $ \out(k) $ and $ \ind_n(m) $ are independent.

 \end{itemize}
 
Note that combining independence structures \texttt{\textbf{Ia}} and \texttt{\textbf{Ic}} implies that for any $ m,k $ with $m>k>i$ the existence of edge $ (m,k) $ is independent of the out- as well as of the indegree of vertex $ i $.  We  will now give some first order properties of $  f(\ind_n(i))$.

\begin{lemma}[Lemma 3.1 in \cite{BDO18}]\label{le:moment}
		For the preferential attachment model defined above with $f(k) \leq \gamma k + 1$ for all $k$, and some $\gamma \in (0,1)$, we have for all $n \in \N$,
	\[ \E[ f(\textnormal{\deg}^-_n(i))] \leq  \Big( \frac{n}{i} \Big)^\gamma \quad \mbox{for all } i \in [n] . \]
\end{lemma}

In the remainder of this subsection we will state some of the results that have been established in \cite{Dereich2009} and \cite{Dereich2013} for the preferential attachment model described above and which will turn out to be beneficial for the proof of our main result.
Conditioning on  vertex $ m $  having a certain indegree at a specific point in time clearly influences the evolution of  the indegree process $ (\deg^-_n(m) )_{n\geq m }$ for all times $ n \geq m $. Lemma~\ref{le:DM 2.8} shows that this influence can be bounded, where the bound depends on the attachment function $ f $. Lemma~\ref{le:StochasticDomination} shows that the degree process, which is known to be in state $ k+1$  for some $ k \in \N_0 $ at time $ m $, stochastically dominates the process which is known to start in $ k $ at the same time conditioned on gaining an edge at some later point in time. Intuitively speaking, this means that the earlier an edge enters the network, the bigger its influence on the emergence of new edges. Note that the connection probability of two vertices only depends on the point in time the younger of the two vertices enters the network and the indegree of the older of the two at that time. This implies that, given its indegree the birthtime of the older vertex is irrelevant for the probability of connecting the two, i.e. 
\begin{align}\label{Eq:OnlyState}
\P((n,\ell)\in \cG\vert \ind_{m}(\ell)=k)=\P((n,r)\in \cG\vert \ind_{m}(r)=k)
\end{align}
  for all $\ell, r \leq m < n $. In particular, this equality holds for $ r= m. $ In order to shorten notation, we will sometimes denote the probability in \eqref{Eq:OnlyState} by  $ \P^{(k)}((n,m) \in \cG) $. 
\begin{lemma}[Lemma 2.8 in \cite{Dereich2013}] \label{le:DM 2.8}
	For an attachment rule $ f $ and integers $ k,m,n\in \N_0 $ with $ 0< m \leq n $ one has
	\begin{equation}
		\frac{\E^{k+1}[f(\textnormal{\deg}^-_n(m))]}{\E^k[f(\textnormal{\deg}^-_n(m))]} \leq \frac{f(k+1)}{f(k)},
	\end{equation}
	where  $ \E^k $ denotes the expectation with respect to the process $ (\deg_n^-(m))_{n\geq m} $ conditional on $ \deg_m^-(m)=k $ (i.e.  with respect to the measure $ \P^{(k)} $). 
	If $ f $ is linear and $ f(k+1+\ell)\leq m+\ell $ for all $ \ell \in \{0,\ldots,n-m-1\} $,
	then equality holds.
\end{lemma}
\begin{lemma}[Lemma 2.10 in \cite{Dereich2013}]\label{le:StochasticDomination}
	 For integers $ 0 \leq k < m < n $, there exists a coupling of the process $ (\textnormal{\deg}_\ell^-(m) : \ell \geq m)  $ started in $ \textnormal{\deg}_m^-(m)= k $ and conditioned on $ \textnormal{\deg}_{n+1}^-(m) -\textnormal{\deg}_n^-(m)=1 $ and the unconditional process $ (\textnormal{\deg}_\ell^-(m) : \ell \geq m)  $ started in $ \textnormal{\deg}_m^-(m)= k +1$, such that for the coupled random evolutions, say $ (Y^{(1)}(\ell):\ell \geq m) $ and $ (Y^{(2)}(\ell):\ell \geq m) $, one has
	 \[
	\Delta Y^{(1)}(\ell)\leq \Delta Y^{(2)}(\ell) + 1\{\ell = n\},
	 \]
 and therefore in particular $ Y^{(1)}(\ell) \leq Y^{(2)}(\ell) $ for all $ \ell \geq m $.
\end{lemma}

\subsection{Stein's method and size-bias coupling}
The main tool used to prove  Theorem~\ref{Thm:isolatedvertices} is to apply  Theorem 1.1 in \cite{GR96}, which uses Stein's method in combination with a size-bias coupling to give a general bound for the approximation of a properly rescaled random variable by a normal distribution. In Theorem~\ref{Thmsizebias} we state a slightly modified version of it, which has already been adapted to the context of random graphs. Before we do so, we recall the definition of size-bias distributions.
\begin{definition}\label{def:size-bias}
	For a random variable $ X\geq 0 $ with $ \E\left[X\right]=\mu < \infty $, we say that the random variable $ X^s $ has the \textbf{size-bias distribution} with respect to X if for all f such that $ \E\left[Xf(X)\right]<\infty $ we have 
	\begin{equation}\label{eq:sizebiaseq}
	\E\left[Xf(X)\right]=\mu \E\left[f(X^s)\right].
	\end{equation}
\end{definition}
For a discrete $ \mathbb{N}_0 $-valued random variable $ X $  Equation \eqref{eq:sizebiaseq} is equivalent to 
\begin{equation}\label{eq.SizeBiasAlternative}
\P(X^s=k)=\frac{k \,\P(X=k)}{\mu} \qquad \forall k \in \mathbb{N}_0.
\end{equation}
This identity nicely illustrates that the size-bias distribution is indeed the original distribution biased by the size of the random variable. 
As stated before, the following result is a slight modification of \cite[Theorem 1.1]{GR96} adapted to the context of random graphs. The proof is identical to the proof given in \cite{GR96}, except for conditioning on the whole graph $ \cG_n $ instead of $ W_n $.

\begin{theorem}\label{Thmsizebias}
	For a random graph $ \mathcal{G}_n $ let $ W_n\geq 0 $ be some $ \sigma(\mathcal{G}_n) $-measurable random variable with $ W_n \geq 0  $, $ \E\left[W_n\right]=\mu_n < \infty $ and $\Var (W_n)=\sigma_n^2 $. Let $W_n^s $ be defined on the same space as $ W_n $ and have the size-bias distribution with respect to $ W_n $. If $ \widetilde{W}_n
	=\frac{W_n-\mu_n}{\sigma_n} $ and $ Z\sim \mathcal{N}(0,1) $, then
	\begin{equation}
	d_W(\widetilde{W}_n,Z)\leq \frac{\mu_n}{\sigma_n^2}\sqrt{\frac{2}{\pi}}\sqrt{\Var(\E\left[W_n^s-W_n\vert \mathcal{G}_n\right])}+\frac{\mu_n}{\sigma_n^3}\E\left[(W_n^s-W_n)^2\right].\label{Thmbound}
	\end{equation}
\end{theorem}
If $ W_n=\sum_{i=1}^{n}X_i $ with $ X_i\geq 0 $ and $ \E\left[X_i\right]=\nu_i $, \cite{GR96} as well as \cite[Section 3.4.1]{Ross2011} provide the following construction of a size-bias version of $ W_n $:
\begin{enumerate}
	\item For each $ i=1, \ldots n $, let $ X_i^s $ have the size-bias distribution of $ X_i $ independent of $ (X_j)_{j \neq i} $ and $ (X^s_j)_{j \neq i} $. Given $ X_i^s=x $, define the vector $ (X_j^{(i)})_{j \neq i} $ to have the distribution of $ (X_{j})_{j \neq i} $ conditional on $ X_i=x $.
	\item Choose a random summand $ X_I $, where the index $ I $ is chosen proportional to $ \mu_i $ and independent of everything else. Specifically, we have $ \P(I=i)=\frac{\nu_i}{\mu_n} $, where $ \mu_n=\E\left[W_n\right] $.
	\item Define $  W_n^s=\sum_{j\neq I}X_j^{(I)}+X_I^s $.
\end{enumerate}

For the special case of Bernoulli random variables $ X_i $, the random variable $ X_i^s \equiv 1 $ has the size-bias distribution of $ X_i $ (see for example \cite[Section 2.2]{AGK2019}), so that with the construction above we obtain the following result (see \cite[Corollary 3.24]{Ross2011}):
\begin{prop}\label{cor:sizebiasBernoulli}
	Let $ X_1, \ldots X_n $ be zero-one random variables and let $ p_i:=\P(X_i=1) $. For each $ i=1, \ldots, n $ let $ (X_j^{(i)})_{j \neq i} $ have the distribution of $ (X_j)_{j \neq i} $ conditional on $ X_i=1 $. If $ W_n=\sum_{i=1}^{n}X_i,\, \mu_n=\E\left[W_n\right] $, and $ I $ is chosen independent of all else with $ \P(I=i)=\frac{p_i}{\mu_n} $, then $ W_n^s=\sum_{j\neq I}X_j^{(I)}+1 $ has the size-bias distribution of $ W_n $.
\end{prop}
\section{Main result}~\label{chapter:main result} 
We consider the distribution of the number of isolated vertices in the preferential attachment model introduced in the previous section.
Here, we call a vertex isolated if it has neither incoming nor outgoing edges. We show that for a certain class of attachment functions this random variable fulfils a central limit theorem. More precisely, we show the following theorem:
\begin{theorem}
	\label{Thm:isolatedvertices}
 Denote by $ W_n $ the number of isolated vertices in the preferential attachment graph $ \cG_n $ described in Section~\ref{sec:model}. For attachment functions $ f $ with $ f(k)\leq \gamma k+1 $ for all $ k \in \N_0 $, 
 some  $ \gamma \in (0,\frac{1}{2}) $ and  $ f(0) < 1 $, there exists a constant $ C \in (0, \infty ) $ such that
	\begin{equation}\label{Eq:mainThm}
		d_W(\widetilde{W}_n,Z)\leq C  \begin{cases}
		\frac{1}{ \sqrt{n}}&\text{ for }\gamma \in (0,\frac{1}{4}),\\[0.3cm]
	\sqrt{	\frac{\log(n)}{n}}&\text{ for }\gamma =\frac{1}{4},\\[0.3cm]
			n^{2\gamma-1}&\text{ for }\gamma \in (\frac{1}{4}, \frac{1}{2}).\\
		\end{cases}
	\end{equation}

\end{theorem}
\begin{remark}
	\begin{itemize}
		\item[(i)] The fact that Theorem~\ref{Thm:isolatedvertices} only yields convergence to a standard Gaussian distribution for parameters $ \gamma < \frac{1}{2} $  is due to the covariance and second moment bounds given in Lemma~\ref{le:Rni}. Though we do not show that these are tight, we do think that there might be a phase transition for the validity of a central limit theorem and that it might not hold true for attachment functions with very strong preference.  This would be in line with the situation in the Erdös-Rényi random graph, where the properly rescaled number of isolated vertices converges in distribution to a standard Gaussian distribution if and only if  $n^2p\rightarrow \infty$ and $np-\log(n)\rightarrow - \infty$, cf. \cite[Theorem 8]{Barbour89}. Also, the authors in \cite{Dereich2013} show that for linear attachment functions  a robust giant component exists if and only if  $ \gamma \geq \frac{1}{2} $, which shows that at least in the linear case the global network structure undergoes a phase transition at $ \gamma =\frac{1}{2}. $ This strengthens the conjecture of the existence of  a phase transition for the distribution of the number isolated vertices, however this is not covered in the present paper and is an open question to be dealt with.
		\item[(ii)]	Attachment functions fulfilling the assumptions in Theorem~\ref{Thm:isolatedvertices} and which define a random graph for which the bound on the right-hand side in \eqref{Eq:mainThm} tends to zero are for instance functions $f:\N\rightarrow \R$ of the form
		\begin{itemize}
			\item[(1)] $ f(k)=\alpha k +\beta $ with $ \alpha < \frac{1}{2} $.
			\item[(2)] $ f(k)=\alpha(k+1)^{\tau} $ with $ \alpha < 1 $ and $ \tau <\frac{\log(3)-\log(2\alpha)}{\log(2)} $.
			\item[(3)]  $ f(k)= \alpha \log(k+1)+\beta $ with $ \beta\in (0,1)$ and $ \alpha<\frac{1}{2\log(2)} $.
		\end{itemize}
	\end{itemize}

\end{remark}
\section{Size-bias construction}\label{sec:SBconstruction}
For every $ i \in [n] $ we construct a random graph $ \cG_{n}^{(i)} $ on $ n $ vertices in which vertex $ i $ is isolated. We will  then couple its evolution to the evolution of $ \cG_{n} $ such that in the coupled graph $ \bar{\cG}_n^{(i)} $ the distribution of the number of isolated vertices is given by the size-bias distribution of the number of the same quantity in $ \cG_{n} $ while at the same time the two random graphs are close in a certain sense. More precisely, for Bernoulli random variables $ X_{k,n} $ and $X_{k,n}^{(i)}$ which equal one if vertex $ k $ is isolated in $ \cG_n $ or $ \bar{\cG}_{n}^{(i)} $, respectively, we will construct $ \bar{\cG}_{n}^{(i)} $ in such a way that 
\begin{equation}\label{eq:conditionsizebias}
\P(X_{k,n}^{(i)}=1)=\P(X_{k,n}=1\vert X_{i,n}=1).
\end{equation}
\subsection{Construction of $ \cG_{n}^{(i)} $}
We construct $ \cG_{n}^{(i)} $ in basically the same way as $ \cG_{n} $ with only minor changes in the connection probabilities. More precisely, let $ f $ be an attachment function as introduced in Section~\ref{sec:model}. 
 We start with $ \cG_{1}^{(i)} $ consisting of one vertex and no edges. At each discrete time step $ n $ we now insert vertex $ n $ into the network and connect it to any older vertex $ k \in [n-1] $ according to the rule
\begin{align}\label{connectionprobGi}
\P((n,k) \in \cGi_n \vert \di_{n-1}(k)=j)= \begin{cases}
\frac{f(j)}{n-1} \cdot  p_{n,j}^{(i)}& \text{ for } k <i,\\[0.2cm]
0 & \text{ for } k=i,\\[0.2cm]
\frac{f(j)}{n-1} &\text{ for } k> i,
\end{cases}
\end{align}
where $  p_{n,j}^{(i)}:=\frac{\P((i,k)\notin \cG_i \vert \d_n(k)=j+1)}{\P((i,k)\notin \cG_i \vert \d_{n-1}(k)=j)} . $\\
Here, 
we introduced the notation $ \d_{n}(k)=\ind_{n}(k) $ and $ \di_{n}(k)=\deg_{n}^{-,(i)}(k) $ for the indegree of vertex $ k $ in $\cG_{n}$ and $ \cG_{n}^{(i)}$, respectively. Note that for both graphs $ \cG_n $ and $ \cGi_n $ edges cannot be removed once inserted into the network, and at time $ n $ edges $ (m,k) $ cannot be added to the network for any $ n>m>k $. Thus $(m,k) \in \cG_n \Leftrightarrow (m,k) \in \cG_m$ and it does not matter which index we use, so that for ease of notation we will omit it whenever there is no need to include it. Lemma~\ref{lem:degreedistributions} now shows that the indegree distributions of vertices in $ \cG_n^{(i)} $ equal the conditional indegree distributions in $ \cG_n $ given that vertex $ i $ is isolated.
\begin{lemma}\label{lem:degreedistributions}
	Fix $ i \in [n] $. For any $ k\in [n], m \in\{k,\ldots, n\}  $ and $ j \in \{0,\ldots, m-k\} $ we have
	\begin{align}\label{eq:degreedistributions}
 \P(\di_{m}(k)=j)=\P(\d_{m}(k)=j \vert X_{i,n}=1)
	\end{align}
\end{lemma}
\begin{proof}
	 We first consider the case  $ m,k > i $. In this situation the rules for building new edges in $ \cG_n $ and $ \cG_n^{(i)} $ are identical, so that
	 \begin{align*}
	 \P(\di_m(k)=j)=\P(\d_m(k)=j)=\P(\di_m(k)=j\vert X_{i,n}=1)
	 \end{align*}
	 since the isolation of vertex $ i $ does not influence the indegree of vertices emerging later than time $ i $ (cf. independence structures \texttt{\textbf{Ia}} and \texttt{\textbf{Ic}}). For $ k=i $ and all $ m\in \{i,\ldots n\} $ we have
	\[
	\P(\di_m(i)=0)=1 =\P(\d_m(i)=0\vert X_{i,n}=1).
	\]
We are now left to deal with the case $ k \in \{1,\ldots, i-1\} $. Note that in this situation
 \begin{equation}\label{eq:reducedcondition}
 \P(\d_{m}(k)=j \vert X_{i,n}=1)= \P(\d_{m}(k)=j \vert (i,k)\notin \cG),
 \end{equation}
 see also independence structure \texttt{Ia}.
 	We will prove the claim via induction on $ m $. Before we do so, note that by the construction of $ \cG^{(i)} $ we have that for any $ k<m<i $ 
 \begin{align*}
& \P((m,k)\in \cG\vert (i,k) \notin \cG, \d_{m-1}(k)=j) \notag \\
  &=  \P((m,k)\in \cG\vert \d_{m-1}(k)=j) \ \frac{ \P((i,k) \notin \cG\vert \d_{m-1}(k)=j, (m,k)\in  \cG )}{\P((i,k)\notin \cG\vert \d_{m-1}(k)=j)}\notag\\[0.25cm]
 &= \P((m,k)\in \cG \vert \d_{m-1}(k)=j) \ \frac{ \P( (i,k)\notin \cG \vert \d_{m}(k)=j+1)}{\P((i,k)\notin \cG \vert \d_{m-1}(k)=j)}\notag\\
 &=\P((m,k)\in \cGi \vert \di_{m-1}(k)=j)
 \end{align*}
 and, for $ m>i $, 
 \begin{align}\label{eq:Connectionm>i}
 & \P((m,k)\in \cG\vert (i,k) \notin \cG, \d_{m-1}(k)=j) = \frac{f(j)}{m-1}= \P((m,k)\in \cGi\vert\di_{m-1}(k)=j).
 \end{align}
so that 
  \begin{align}\label{eq:condprob}
 & \P((m,k)\in \cG\vert (i,k) \notin \cG, \d_{m-1}(k)=j) = \P((m,k)\in \cGi\vert  \di_{m-1}(k)=j)  
 \end{align}
  for all $ m\in \{k+1,\ldots n\} $. We are now set to prove \eqref{eq:degreedistributions}.
  Since the constructions of both graphs do not allow for loops, we have $ \d_k(k)=\di_k(k)=0 $ almost surely and thus the statement is clear for $ m=k $. For $ m=k+1 $ we obtain
 \begin{align*}
 \P(\di_{k+1}(k)=0)&=\P((k+1,k)\notin \cGi\vert \di_k(k)=0)\\
 &= \P((k+1,k)\notin \cG\vert (i,k)\notin \cG, \d_{k}(k)=0) \\
 & =\P(\d_{k+1}(k)=0\vert (i,k)\notin \cG),
  \end{align*}
 where we used   
 the fact that according to \eqref{eq:condprob}
 \begin{align*}
 \P((k+1,k)\notin \cG\vert (i,k) \notin \cG, \d_{k}(k)=0)
 &= \P((k+1,k)\notin \cG^{(i)}\vert\di_{k}(k)=0)
 \end{align*}
  and  $\d_k(k) =0$ a.s.. Moreover,
  \begin{align*}
 \P(\di_{k+1}(k)=1)&=1-\P(\di_{k+1}(k)=0)=1-\P(\d_{k+1}(k)=0\vert (i,k)\notin \cG )\\
 &= \P(\d_{k+1}(k)=1\vert (i,k)\notin \cG),
  \end{align*}
 so that the claim holds for $ m=k+1 $ and $ j \in \{0,1\}. $ Assume now that \eqref{eq:degreedistributions} holds for  $ m-1 \in \{1, \ldots k-1\}$ and all $ j \in \{0,\ldots, m-k-1\}. $ Then, by construction of $ \cG^{(i)} $,
 \begin{align*}
 \P(\di_m(k)=j)&=\P((m,k)\notin \cGi\vert \di_{m-1}(k)=j)\P(\di_{m-1}(k)=j)\\[0.2cm]
 &\quad + \P((m,k)\in \cGi \vert \di_{m-1}(k)=j-1)\P(\di_{m-1}(k)=j-1)
 \end{align*}
for any $ j \in \{0,\ldots, m-k-1\}. $ Using the induction hypothesis and  \eqref{eq:condprob} twice yields
 \begin{align*}
  \P&(\di_m(k)=j)\\[0.25cm]
  &=  \P((m,k)\notin \cG\vert (i,k)\notin \cG, \d_{m-1}(k)=j)\, \P(\d_{m-1}(k)=j\vert (i,k)\notin \cG)\\
  &\quad +\P((m,k)\in \cG\vert (i,k)\notin \cG, \d_{m-1}(k)=j-1)\,\P(\d_{m-1}(k)=j-1\vert (i,k)\notin \cG)\\[0.25cm]
   & = \P(\d_{m-1}(k)=j, (m,k)\notin \cG\vert (i,k)\notin \cG)+ \P(\d_{m-1}(k)=j-1, (m,k)\in \cG \vert (i,k)\notin \cG)\\[0.25cm]
  &= \P(\d_{m-1}(k)=j,\d_{m}(k)=j  \vert (i,k)\notin \cG)+ \P(\d_{m-1}(k)=j-1, \d_{m}(k)=j\vert (i,k)\notin \cG)\\[0.25cm]
&  =\P(\d_m(k)=j\vert (i,k)\notin \cG)
 \end{align*}
for $ j \in \{0,\ldots, m-k-1\} .$ It remains to show that \eqref{eq:degreedistributions} also holds for $ j=m- k$. Using \eqref{eq:condprob} again shows that
  \begin{align*}
 \P(\di_m(k)=m-k)&=\prod_{r=k+1}^{m}\P((r,k)\in \cGi \vert \di_{r-1}(k)=r-k-1)\\&=\prod_{r=k+1}^{m}\P((r,k)\in \cG\vert (i,k)\notin \cG, \d_{r-1}(k)=r-k-1)\\
 &=\P\,\Big(\bigcap_{r=k+1}^m\big\{(r ,k) \in \cG\big\} \vert (i,k)\notin \cG\Big)= \P(\d_m(k)=m-k\vert (i,k)\notin \cG).
 \end{align*}
 This completes the proof for $ k \in \{1,\ldots, i-1\} $ and all $ m\geq k $, thus proving the Lemma.
\end{proof}
\begin{lemma}\label{lem:connectionprob}
	For a random graph $ \cG_n^{(i)} $ on $ n $ vertices constructed as outlined at the beginning of this section and a random graph $ \cG_n $ build according to the construction rules given in Section~\ref{sec:model} we have that for any $ m,k \in [n] $ with $ m >k $
	\begin{align}\label{eq:connectionprob}
	\P((m,k)\in \cG^{(i)})=\P((m,k)\in \cG \vert X_{i,n}=1).
	\end{align}
\end{lemma}
\allowdisplaybreaks
\begin{proof}
	For $ m,k \geq i $ the statement is clear due to the construction rules of $ \cG $ and $ \cGi $ given in \eqref{connectionprob} and \eqref{connectionprobGi}, respectively. Note that for $ i>m $
	\[
	\P((i,k)\notin \cG\vert \d_{m}(k)=\ell+1)=\P((i,k)\notin \cG\vert \d_{m-1}(k)=\ell, (m,k)\in \cG)
	\] 
 since connection probabilities depend on the indegree of the older vertex at the time of insertion of the younger vertex, but not on the wohle degree evolution. Combining this observation with
 Lemma~\ref{lem:degreedistributions} and Equation \eqref{connectionprobGi} 	 yields
	\begin{align*} 
	\P((m,k)\in \cGi)&=\sum_{\ell=0}^{m-k-1}\P((m,k)\in \cGi\vert \di_{m-1}(k)=\ell)\P(\di_{m-1}(k)=\ell)\\
	&=\sum_{\ell=0}^{m-k-1}\P(\d_{m-1}(k)=\ell\vert X_{i,n}=1) \ \frac{f(\ell)}{m-1} \, p_{m,\ell}^{(i)}\\
		& =\sum_{\ell=0}^{m-k-1}\P(\d_{m-1}(k)=\ell\vert X_{i,n}=1)\P((m,k)\in \cG\vert \d_{m-1}(k)=\ell)\\
		&\qquad \qquad  \times\frac{\P((i,k)\notin \cG\vert \d_{m-1}(k)=\ell, (m,k)\in \cG)}{\P((i,k)\notin \cG\vert \d_{m-1}(k)=\ell)}\\
		&=\sum_{\ell=0}^{m-k-1}\P(\d_{m-1}(k)=\ell\vert X_{i,n}=1)\P((m,k)\in \cG\vert (i,k)\notin \cG, \d_{m-1}(k)=\ell)\\
		&=\sum_{\ell=0}^{m-k-1}\P(\d_{m-1}(k)=\ell\vert X_{i,n}=1)\P((m,k)\in \cG\vert X_{i,n}=1, \d_{m-1}(k)=\ell)\\
		& =\P((m,k)\in \cG\vert X_{i,n}=1),
	\end{align*}
		for $ k <m<i $. For $ m>i $ we combine Lemma~\ref{lem:degreedistributions} with Equation Stochastic geometry to generalize the Mondrian Process, joint with Ngoc Tran, to appear in SIAM Journal on Mathematics of Data Science, 
		 \eqref{eq:Connectionm>i} to obtain 
	\begin{align*}
		\P((m,k)\in \cGi)&=\sum_{\ell=0}^{m-k-1}\P((m,k)\in \cGi\vert \di_{m-1}(k)=\ell)\P(\di_{m-1}(k)=\ell)\\
	&=\sum_{\ell=0}^{m-k-1}\P((m,k)\in\cG\vert \d_{m-1}(k)=\ell, (i,k)\notin \cG)\P(\d_{m-1}(k)=\ell\vert X_{i,n}=1)\\
		&=\sum_{\ell=0}^{m-k-1}\P((m,k)\in\cG\vert \d_{m-1}(k)=\ell,X_{i,n}=1)\P(\d_{m-1}(k)=\ell\vert X_{i,n}=1)\\
		&=\P((m,k)\in \cG_n\vert X_{i,n}=1).
	\end{align*}
\end{proof}
\subsection{Coupling of $ \cG_n $ and $ \cG_{n}^{(i)} $}\label{sec:coupling}
In this section we couple the degree evolutions of  $ \cG_{n}^{(i)} $ and $ \cG_n$. 
One can think of this coupling as a two stage process in which $ \cGi_n  $ is constructed from $ \cG_n $  in the following way:
we start by building the graph $ \cG_n $ according to the construction given in Section~\ref{sec:model}. Based on the whole evolution of $  \cG_n $ we can now successively construct $ \bar{\cG}_n^{(i)} $  for fixed $ i \in [n] $ in the following way:
starting with $ \bar{\cG}_1^{(i)} $ consisting of a single vertex and no edges, for every $ m \in [n]$ we construct $ \bar{\cG}_{m}^{(i)} $ based on  $ \bar{\cG}_{m-1}^{(i)} $ and the whole evolution of $ \cG_n $ according to the following rules:
\begin{itemize}
	\item[(i)] for any $ k \in [n]$ and $ m \in\{k+1,\ldots, n\} $: $(m,k) \notin \cG \Rightarrow (m,k) \notin \bar{\cG}^{(i)}$.
	\item[(ii)] at time $ i $ no edges are inserted into the network $  \bar{\cG}^{(i)}  $, i.e. $ (i,k) \notin  \bar{\cG}^{(i)}  $ for all $ k \in\{1,\ldots, i-1\} $. 
	\item[(iii)] for any $ m\in\{i+1,\ldots,n\} $ edge $ (m,i) $ is not inserted into the network $  \bar{\cG}^{(i)}  $, regardless of whether it exists in $ \cG $, i.e.$ (m,i) \notin \bar{\cG}^{(i)}$ for all $ m\geq i .$
	\item[(iv)] for $ m>k > i $:  $ (m,k) \in \bar{\cG}^{(i)} \ \Leftrightarrow \ (m,k) \in \cG.$
		\item[(v)] for any $ k \in [i-1], m\geq k+1 $ and  $ (i,k) \notin  \cG $: $ (m,k) \in \bar{\cG}^{(i)} \ \Leftrightarrow \ (m,k) \in \cG.$
	\item[(vi)]   for any $ k \in [i-1], m\geq k+1 $ and  $ (i,k) \in  \cG $: if $ (m,k) \in \cG  $ and  $ \d_{m-1}(k)=j $ and $ \di_{m-1}(k)=\ell $, $(m,k)\in \bar{\cG}^{(i)}$ with probability
	\begin{align*}
	\pi_m^{(i)} (j,\ell)
	= \frac{f(\ell)}{f(j)} \begin{cases}
	p_{m,\ell}^{(i) }\cdot \frac{\P((i,k)\in \cG\vert \d_{m-1}=j)}{\P((i,k)\in \cG\vert \d_{m}=j+1)}& \text{ for } m <i,\\
	1& \text{ for } m >i.
	\end{cases}
	\end{align*}

\end{itemize}
\begin{remark}
	\begin{itemize}\, 
		\item[(i)]Note that incoming edges of vertex $ k \in \{1, \ldots, i-1\} $ depend on the existence of edge $ (i,k) $ in $ \cG $  which is why we need to construct  $ \cG $ (at least up to time $ i $) first before we can couple the evolution of $ \cG^{(i)} $ to it.
		
		\item[(ii)]According to the first item above the edge set of $ \bar{\cG}^{(i)} $ is a subset of that of $ \cG $ and one can think of this construction as building $ \cG^{(i)} $ based on $ \cG $ by reconsidering edges present in $ \cG$ that have been affected by the isolation of vertex $ i $ (so that we only reconsider incoming edges of vertices that were connected to vertex $ i $ in $ \cG $). Figure~\ref{Fig:coupledgraph} illustrates which edges are affected by the isolation of a vertex in this procedure.
		\begin{figure}[t]
			\centering
			\includegraphics[scale=0.5]{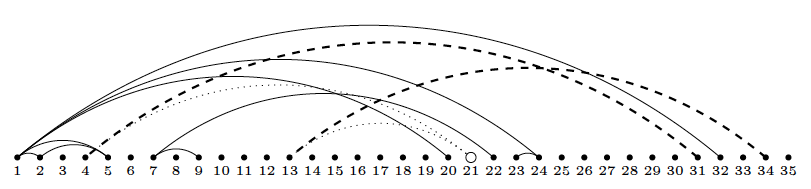}
			\caption{\small Edges deleted in $\bar{\cG}^{(21)} $ are depicted as dotted lines, edges affected and possibly deleted by the isolation are depicted as dashed lines. The solid lines represent edges not affected by the isolation of vertex 21 and are thus included in $ \bar{\cG}^{(21)} $.}
			\label{Fig:coupledgraph}
		\end{figure}
		
	\end{itemize}
\end{remark}

The calculations below now show that the probabilities of connecting vertices in $ \cGi_n $ and $ \bar{\cG}_n^{(i)} $ are identical. For ease of notation we define the event 
\[
\texttt{D}^{(i)}_{m}(k;j,\ell)= \{\cd_{m}(k)=j\} \cap \{\cdi_{m}(k)=\ell\}
\]
for $ j, \ell \in [n] $ and $ m >k $. Note that for $  j \neq \ell $ and $ m <i $
\begin{align*}
\P((m,k)\in \cG\vert \texttt{D}^{(i)}_{m-1}(k;j,\ell))= \P((m,k)\in \cG\vert \d_{m-1}(k)=j, (i,k)\in \cG )
\end{align*}
since the indegree of vertex $ k $ in $ \cG $ and $\bar{\cG}^{(i)}$ can only differ if $ (i,k) \in \cG $ and given this information and $ \{\d_{m-1}(k)=j\} $, the event $ \{\cdi_{m}(k)=\ell\} $ 
does not contribute any further information for the connection probability of $ m $ and $ k $ in $ \cG $. 
With this considerations and the construction rules (i)-(vi) for $\bar{\cG}$ given above, we obtain for 
any $ k \in [i-1], m \in \{k+1,\ldots,i-1\} $ and $ \ell,j \in \{0,\ldots,m-k-1\} $ with $  j \neq \ell   $, 
\begin{align}\label{eq:EdgeCouledGraphuneq}
&\P((m,k)\in \bar{\cG}^{(i)} \vert \texttt{D}^{(i)}_{m-1}(k;j,\ell))=\P((m,k)\in \bar{\cG}^{(i)}, (m,k)\in \cG \vert \texttt{D}^{(i)}_{m-1}(k;j,\ell))\notag\\[0.25cm]
&=\P((m,k)\in \bar{\cG}^{(i)}\vert  \texttt{D}^{(i)}_{m-1}(k;j,\ell), (m,k)\in \cG )\, \P((m,k)\in \cG\vert \texttt{D}^{(i)}_{m-1}(k;j,\ell))\notag\\[0.25cm]
&=\P((m,k)\in \bar{\cG}^{(i)}\vert  \texttt{D}^{(i)}_{m-1}(k;j,\ell), (m,k)\in \cG, (i,k)\in \cG )\, \P((m,k)\in \cG\vert \d_{m-1}(k)=j, (i,k)\in \cG )\notag\\[0.25cm]
&=\pi_m^{(i)}(j,\ell)\cdot \frac{\P((i,k)\in \cG\vert \d_{m}(k)=j+1)}{\P((i,k)\in \cG\vert \d_{m-1}(k)=j)}\cdot \P((m,k)\in \cG\vert \d_{m-1}(k)=j)\notag\\[0.25cm]
&=\frac{f(\ell)}{m-1} \cdot p_{m}^{(i)}(\ell)= \P((m,k)\in \cGi\vert \di_{m-1}(k)=\ell).
	\end{align}
	With similar arguments as above we get for $ j=\ell $ and $ m<i $
	\begin{align}\label{eq:EdgeCouledGrapheq}
	&\P((m,k)\in \bar{\cG}^{(i)} \vert \texttt{D}^{(i)}_{m-1}(k;\ell,\ell))=\P((m,k)\in \bar{\cG}^{(i)}, (m,k)\in \cG \vert \texttt{D}^{(i)}_{m-1}(k;\ell,\ell))\notag\\[0.25cm]
	&= \P((m,k)\in \bar{\cG}^{(i)}, (m,k)\in \cG\vert \texttt{D}^{(i)}_{m-1}(k;\ell,\ell), (i,k)\in \cG)\, \P((i,k)\in \cG \vert \texttt{D}^{(i)}_{m-1}(k;\ell,\ell))\notag\\
	&\quad + \P((m,k)\in \bar{\cG}^{(i)}, (m,k)\in \cG\vert \texttt{D}^{(i)}_{m-1}(k;\ell,\ell), (i,k)\notin \cG)\, \P((i,k)\notin \cG \vert \texttt{D}^{(i)}_{m-1}(k;\ell,\ell))\notag\\[0.25cm]
	&=\pi_{m}^{(i)}(\ell,\ell)\, \P((m,k)\in \cG\vert \d_{m-1}(k)=\ell, (i,k)\in \cG)\, \P((i,k)\in \cG \vert \texttt{D}^{(i)}_{m-1}(k;\ell,\ell))\notag\\
	&\quad +\P((m,k)\in \cG\vert \d_{m-1}(k)=\ell, (i,k)\notin \cG)\, \P((i,k)\notin \cG \vert \texttt{D}^{(i)}_{m-1}(k;\ell,\ell))\notag\\[0.25cm]
	&=\pi_{m}^{(i)}(\ell,\ell)\,\P((m,k)\in \cG\vert \d_{m-1}=\ell)\,\frac{\P((i,k)\in \cG\vert \d_m=\ell+1)}{\P((i,k)\in \cG\vert \d_{m-1}=\ell)} \P((i,k)\in \cG \vert \texttt{D}^{(i)}_{m-1}(k;\ell,\ell))\notag\\
	&\quad +\,p_{m,\ell}^{(i)}\, \P((m,k)\in \cG\vert \d_{m-1}=\ell)\, \P((i,k)\notin \cG \vert \texttt{D}^{(i)}_{m-1}(k;\ell,\ell))\notag\\[0.25cm]
	&=p_{m,\ell}^{(i)}\, \P((m,k)\in \cG\vert \d_{m-1}=\ell)= \P((m,k)\in \cGi\vert \di_{m-1}(k)=\ell),
	\end{align}
	where we used that  edge $ (m,k) $ needs to be present in $ \bar{\cG}^{(i)} $ if it is present in $ \cG  $
	\begin{align}\label{eq:edgecoupledgraphm<i}
&\P((m,k)\in \bar{\cG}^{(i)}\vert \cdi_{m-1}(k)=\ell)\notag \\
&=\sum_{ j=\ell}^{m-k-1}\P((m,k)\in \bar{\cG}^{(i)}\vert \texttt{D}^{(i)}_{m-1}(k;j,\ell)) \ \P(\cd_{m-1}(k)=j\vert \cdi_{m-1}(k)=\ell)\notag \\
& =\P((m,k)\in \cGi\vert \di_{m-1}(k)=\ell).
\end{align}
For $ m> i $ and $ j \neq \ell $ we have
\begin{align*}
\P((m,k)\in \cG\vert \texttt{D}^{(i)}_{m-1}(k;j,\ell))&= \P((m,k)\in \cG\vert \d_{m-1}(k)=j, (i,k)\in \cG )\\
&=\P((m,k)\in \cG\vert \d_{m-1}(k)=j)
\end{align*}
because, given the indegree of vertex $ k $ at time  $ m-1\geq  i $, the evolution of the indegree up to this time  is irrelevant for the connection probability. Similarly,
\begin{align*}
\P((m,k)\in \cG\vert \texttt{D}^{(i)}_{m-1}(k;\ell,\ell))&= \P((m,k)\in \cG\vert \d_{m-1}(k)=\ell, (i,k)\notin \cG )\\
&= \P((m,k)\in \cG\vert \d_{m-1}(k)=\ell )
\end{align*}
as the event $ \texttt{D}^{(i)}_{m-1}(k;\ell,\ell) $ implies that $ (i,k)\notin \cG  $ because otherwise the indegrees of $ k $ in $ \cG $ and $ \bar{\cG}^{(i)} $ would differ by at least one from time $ i $ on. Also, in this case edge $ (m,k) $ is present in $\bar{\cG}^{(i)}$ if and only if it is present in $ \cG $,  ie.
\begin{align*}
\P((m,k)\in \bar{\cG}^{(i)}\vert \texttt{D}^{(i)}_{m-1}(k;\ell,\ell))&=\P((m,k)\in \bar{\cG}^{(i)}, (m,k)\in \cG\vert \texttt{D}^{(i)}_{m-1}(k;\ell,\ell))\\
&= \P((m,k)\in \cG\vert \texttt{D}^{(i)}_{m-1}(k;\ell,\ell))=\frac{f(\ell)}{m-1} .
\end{align*}
For $\ell \neq j$ we have
\begin{align}\label{eq:EdgeCouledGraphLargem}
\P((m,k)\in \bar{\cG}^{(i)}\vert\texttt{D}^{(i)}_{m-1}(k;\ell,j)&=\P((m,k)\in \bar{\cG}^{(i)}, (m,k)\in \cG\vert\texttt{D}^{(i)}_{m-1}(k;\ell,j))\notag\\
&= \pi_{m-1}^{(i)}(j,\ell)\P((m,k)\in \cG\vert \d_{m-1}(k)=j)
\end{align}
so that for $ m>i  $ we obtain
\begin{align}\label{DegreesCoupledGraph}
&\P((m,k)\in \bar{\cG}^{(i)}\vert \cdi_{m-1}(k)=\ell)\notag\\
&\quad = \sum_{ j=\ell}^{m-k-1}\P((m,k)\in \bar{\cG}^{(i)}, (m,k) \in \cG\vert \texttt{D}^{(i)}_{m-1}(k;j, \ell))\P(\cd_{m-1}(k)=j\vert \cdi_{m-1}(k)=\ell) \notag\\
&\quad = \sum_{ j=\ell+1}^{m-k-1}\hspace{-2ex}\pi_{m-1}^{(i)}(j,\ell)\frac{f(j)}{m-1}\P(\cd_{m-1}(k)=j\vert \cdi_{m-1}(k)=\ell)\hspace{-0.25ex}+\hspace{-0.25ex}\frac{f(\ell)}{m-1}\P(\cd_{m-1}(k)=\ell\vert \cdi_{m-1}(k)=\ell) \notag\\
&\quad =\frac{f(\ell)}{m-1}=\P((m,k)\in \cGi\vert \di_{m-1}(k)=\ell).
\end{align}
Consequently, 
\[
\P((m,k)\in \bar{\cG}^{(i)})= \P((m,k)\in \cG^{(i)})
\]
and by proceeding in the same way as in the proof of Lemma~\ref{lem:degreedistributions}, we see that $\cdi_m(k)\stackrel{d}{=}\di_m(k) $ and $ \coi(k)\stackrel{d}{=}\oi(k) $ for any $ m,k \in [n] $. By a slight abuse of notation we will from now on write $ \di_m(k) $  and $ \oi(k)  $ when referring to the in- and outdegrees of vertices in the coupled graph.

\begin{prop}\label{prop:deletionprob}
	Fix $ i \in [n] $. Let $ \cG $ and $ \bar{\cG}^{(i)} $ be the coupled graphs as described above with attachment function $ f $ with $ f(k)\leq \gamma k +1 $ for all $ k \in \N_0  $ and some $ \gamma \in (0,1) $ 
	\begin{align*}
	\P((m,k)\notin \bar{\cG}^{(i)}, (m,k)\in \cG)&= \P((m,k)\in \cG)-\P((m,k)\in \cGi)\lesssim (mi)^{\gamma-1}k^{-2\gamma}.
	\end{align*}
\end{prop}
\begin{proof}
	Note that
	\begin{align*}
	\P((m,k)\notin \bar{\cG}^{(i)}, (m,k)\in \cG )&=\P((m,k)\in \cG)- \P((m,k)\in \bar{\cG}^{(i)}, (m,k)\in \cG )\\
	&= \P((m,k)\in \cG)-\P((m,k)\in \bar{\cG}^{(i)})\\
	&= 	 \P((m,k)\in \cG)-\P((m,k)\in \cG^{(i)})
	\end{align*}
	according to \eqref{DegreesCoupledGraph}.

	To prove the inequality, note that due to the construction of our coupling, we have 
	\[
	\{(m,k)\notin \bar{\cG}^{(i)}, (m,k)\in \cG \} \subset \{(i,k)\in \cG, (m,k)\in \cG \},
	\]
	so that by Lemmas~\ref{le:DM 2.8} and \ref{le:moment}
	\[
		\P((m,k)\notin \bar{\cG}^{(i)}, (m,k)\in \cG )\leq \P((i,k)\in \cG, (m,k)\in \cG) \leq \frac{f(1)}{f(0)}(mi)^{\gamma-1} k^{-2\gamma}.
	\]
\end{proof}
A straightforward result of the previous proposition is the following bound on the expected difference of the indegrees in the two graphs.
\begin{corollary}\label{le:indegreediff}
	Fix $ i \in [n] $ and let $ \Delta^{(i)}_{m} \textnormal{\texttt{d}}(k):=\ind_{m}(k)-\indi_m(k). $ 
	Then, for every $ k \in [i-1] $ and any $ m \geq k $,  we have
	\begin{align*}
	\E[\Delta_m^{(i)}\textnormal{\texttt{d}}(k)]\lesssim    m^\gamma i^{\gamma-1}k^{-2\gamma}.
	\end{align*}
\end{corollary}
\begin{proof}
	We have 
	\begin{align*}
	\E[\Delta_m^{(i)}\texttt{d}(k)]	&=\sum_{r=k+1}^m \P((r,k)\in \cG, (r,k)\notin \cG^{(i)})= \sum_{r=k+1}^m \P((r,k)\in \cG)- \P((r,k)\in \cG^{(i)}),
	\end{align*}
	
	so that by 
	Proposition~\ref{prop:deletionprob} 
	\begin{align*}
	\E[\Delta_m^{(i)}\texttt{d}(k)]	&= \sum_{r=k+1}^m \P((r,k)\in \cG)- \P((r,k)\in \cG^{(i)})\\
	& \lesssim \sum_{r=k+1}^m (ri)^{\gamma-1}k^{-2\gamma} \asymp  m^{\gamma}i ^{\gamma-1}k^{-2\gamma}
	\end{align*}
	according to \eqref{integraltest}.
\end{proof}
The next lemma shows, that the coupling described in this section is such that \eqref{eq:conditionsizebias} holds.
\begin{lemma}
	Let $ \cG_n $ and $ \bar{\cG}_n^{(i)} $ be two random graphs coupled as described above. Furthermore, let $ X_{k,n} $ and $ X^{(i)}_{k,n} $ denote Bernoulli random variables which equal one iff vertex k is isolated in $ \cG_n $ and $ \bar{\cG}_n^{(i)}$, respectively. We then have
	\begin{equation}\label{eq:sizebiascondition}
	\P(X_{k,n}^{(i)}=1)=\P(X_{k,n}=1\vert X_{i,n}=1).
	\end{equation}
\end{lemma}
\begin{proof}
	Remember that we write $ \oi(k)$ for the outdegree of vertex $ k  $ in $ \bar{\cG}_n^{(i)} $. Using  Lemma~\ref{lem:connectionprob}, Equations~\eqref{eq:edgecoupledgraphm<i}  and \eqref{eq:EdgeCouledGraphLargem} and exploiting independence structures \texttt{\textbf{Ia}} and \texttt{\textbf{Ib}} we obtain
	\begin{align*}
	\P(X_{k,n}^{(i)}=1)&=\P(\di_n(k)=0,\oi(k)=0)=\P(\di_n(k)=0)\P(\oi(k)=0)\\
	&=\prod_{r=1}^{k-1}\P((k,r)\notin\bar{\cG}^{(i)}) \prod_{r=k+1}^{n}\P((r,k)\notin \bar{\cG}^{(i)} \vert \di_{r-1}(k)=0)\\
	&=\prod_{r=1}^{k-1}\P((k,r)\notin\cG^{(i)}) \prod_{r=k+1}^{n}\P((r,k)\notin \cG^{(i)} \vert \di_{r-1}(k)=0)\\
	&=\prod_{r=1}^{k-1}\P((k,r)\notin\cG\vert X_{i,n}=1)  \prod_{r=k+1}^{i-1} \left(1-\frac{f(0)}{r-1}p_{r,0}^{(i)}\right)\ \prod_{r=\max\{i+1,k+1\}}^{n}\left(1-\frac{f(0)}{r-1}\right),
	\end{align*}
	For $ k>i $ we have 
	\[
	\P(\d_n(k)=0\vert X_{i,n}=1)=\P(\d_n(k)=0)= \prod_{r=k+1}^{n}\left(1-\frac{f(0)}{r-1}\right)
	\]
	proving the claim in this case. Using \eqref{eq:reducedcondition} and \eqref{Eq:OnlyState}, we see that for $ k <i	 $
	\begin{align}\label{eq:condoutdegree}
	\P(\d_n(k)=0\vert X_{i,n}=1)&= \P(\d_n(k)=0\vert (i,k) \notin \cG)
	\notag\\
	&= \prod_{r=k+1}^{n}\hspace{-0.2cm}\P((r,k)\notin \cG\vert \d_{r-1}(k)=0, (i,k)\notin \cG)\notag\\[0.2cm]
	&= \prod_{\substack{r=k+1}}^{i-1}\hspace{-0.2cm}\frac{\P((r,r-1)\notin \cG)\,\P((i,k)\notin \cG\vert\d_{r}(k)=0)}{\P((i,k)\notin \cG\vert\d_{r-1}(k)=0)}\,\prod_{r=i+1}^{n}\hspace{-0.2cm}\Big(1-\frac{f(0)}{r-1}\Big)\notag\\
	&=\prod_{r=k+1}^{i-1}\hspace{-0.2cm}\Big(1-\frac{f(0)}{r-1}\Big)\frac{\P((i,r)\notin \cG)}{\P((i,r-1)\notin \cG)}\,\prod_{r=i+1}^{n}\hspace{-0.2cm}\Big(1-\frac{f(0)}{r-1}\Big)
	\end{align}
and 
	\begin{align}\label{eq:condindegree1}
	&\Big(1-\frac{f(0)}{r-1}\Big)\frac{\P((i,r)\notin \cG)}{\P((i,r-1)\notin \cG)}=\frac{\P((i,r)\notin \cG)}{\P((i,r-1)\notin \cG)}-\frac{f(0)}{r-1}\frac{\P((i,r)\notin \cG)}{\P((i,r-1)\notin \cG)}\notag\\
\notag	\\
	&=1+\frac{\P((i,r-1)\in \cG)-\P((i,r)\in \cG)-\P((r,r-1)\in \cG)\P((i,r)\notin \cG)}{\P((i,r-1)\notin \cG)}\notag\\
	&=1+\frac{f(0)}{r-1}\frac{\big(\P((i,r-1)\in \cG\vert \d_r(r-1)=1)-\P((i,r)\in \cG)-\P((i,r)\notin \cG))}{\P((i,r-1)\notin \cG)}\notag\\
	&=1-\frac{f(0)}{r-1}\frac{\P((i,r-1)\notin \cG\vert \d_r(r-1)=1)}{\P((i,r-1)\notin \cG)}= 1-\frac{f(0)}{r-1}p_{r,0}^{(i)},
	\end{align}
	so that also in the case $ k<i $
	\begin{align*}
	\P(X_{k,n}^{(i)}=1)	&=\prod_{r=1}^{k-1}\P((k,r)\notin\cG_n\vert X_{i,n}=1) \prod_{r=i+1}^{n}\left(1-\frac{f(0)}{r-1}\right) \prod_{r=k+1}^{i-1} \left(1-\frac{f(0)}{r-1}p_{r,0}^{(i)}\right)\\
	\\	
	&=\P(\o(k)=0\vert X_{i,n}=1) \P(\d_n(k)=0\vert X_{i,n}=1)\\
	&=\P(X_{k,n}=1\vert X_{i,n}=1).
	\end{align*}
\end{proof}
The next lemma now shows, that the number of isolated vertices in the graph $ \bar{\cG}^{(i)}_n $ indeed follows the size-bias distribution of the number of isolated vertices in the original graph $ \cG_n $.
\begin{lemma}
	Let $ W_n=\sum_{k=1}^{n}X_{k,n} $ denote the number of isolated vertices in $ \cG_n $ and set $ \mu_n=\E[W_n] $. Furthermore, let $ W_n^s $ be a random variable having the size-bias distribution of $ W_n $ and denote by $ W_n^{(i)}=\sum_{k=1}^{n}X_{k,n}^{(i)} $ the number of isolated vertices in $ \bar{\cG}_n^{(i)} $. For $ W_n^{(I)} $, where $\P( I=i)= \frac{\P(X_{i,n}=1)}{\mu_n} $, we then have that
	\begin{align*}
	W_n^{(I)}\stackrel{\mathcal{D}}{=}W_n^s.
	\end{align*}
\end{lemma}
\begin{proof}
	For Bernoulli random variables fulfilling \eqref{eq:sizebiascondition} a general proof can for example be found in \cite[Proposition 3.21]{Ross2011}. To make this paper more self-contained we give a slightly adapted proof here. 
	In order to prove the above we show that  equation \eqref{eq:sizebiaseq} in Definition~\ref{def:size-bias} holds with $ X=W_n $ and $ X^s=W_n^{(I)} $. For any $ f $ such that $ \E\left[W_nf(W_n)\right]< \infty $ we have
	\begin{align*}
	\E\left[W_nf(W_n)\right]&=\sum_{i=1}^n\E\left[X_{i,n}f(W_n)\right]=\sum_{i=1}^n\E\left[X_{i,n}\E\left[f(W_n)\vert X_{i,n}\right]\right]\\
	&=\sum_{i=1}^n\P(X_{i,n}=1)\E\left[f(W_n)\vert X_{i,n}=1\right]\\
	&=\sum_{i=1}^n\P(X_{i,n}=1)\E\big[f\big(W_n^{(i)}\big)\bigl]
	\end{align*}
	and 
	\begin{align*}
	\E\big[f\big(W_n^{(I)}\big)\big]=\frac{1}{\mu_n}\sum_{i=1}^n \P(X_{i_n}=1) \E\big[f\big(W_n^{(i)}\big)\big],
	\end{align*}
	so that 
	\begin{align*}
	\E\left[W_nf(W_n)\right]&= \sum_{i=1}^n\P(X_{i,n}=1)\E\big[f\big(W_n^{(i)}\big)\big]=\mu_n\E\big[f\big(W_n^{(I)}\big)\big].
	\end{align*}
\end{proof}

\section{Proof of Theorem~\ref{Thm:isolatedvertices}}\label{sec:proofMainresult} 
To prove our main result we  need to bound the two terms appearing in  \eqref{Thmbound}, i.e.
\begin{align}\label{eq:termstobound}
\frac{\mu_n}{\sigma_n^2}\sqrt{\frac{2}{\pi}}\sqrt{\Var(\E\left[W_n^s-W_n\vert \mathcal{G}_n\right])} \quad
\text{ and }\quad
\frac{\mu_n}{\sigma_n^3}\E\left[(W_n^s-W_n)^2\right],
\end{align}
where $ \mu_n=\E[W_n] $ and $ \sigma_n^2=\Var[W_n] $.

To bound these expressions, note that by the construction of $ W_n^s$ we have
\begin{equation}\label{eq:diff_sizebias}
W_n^s-W_n=D_{n,I} +{\bf 1}\{\texttt{D}_n(I)>0\}+ R_{n,I},
\end{equation}
where $ D_{n,I}=|\mathsf{D}_{n,I}| $ and $ \mathsf{D}_{n,I} $ denotes the set of  neighbours of vertex $ I $ with total degree one (i.e. $I $ is their unique neighbour),  $ \texttt{D}^{(i)}_n(I) $ gives the total degree (i.e. the sum of in- and outdegree) of vertex $ I $, and $ R_{n,I} =|\mathsf{R}_{n,I}|  $, where  $ \mathsf{R}_{n,I} $ refers to the set of vertices that are not in $ \mathsf{D}_{n,I} $ and  which  are isolated in $ \bar{\cG}^{(I)} $ but not in $ \cG. $
 From \eqref{eq:diff_sizebias} we see that in order to bound the terms in \eqref{eq:termstobound} we need to control the first and second order properties of $ W_n $, $ D_{n,I} $ and $ R_{n,I} $. Bounds for these are given in Lemmas~\ref{le:orderEV},~\ref{le:Dni} and \ref{le:Rni}, respectively. With these at hand we  then deduce upper bounds on the two terms given above in Lemmas~\ref{le:variancebound} and \ref{le:squareddistance}, which will be used to prove Theorem~\ref{Thm:isolatedvertices}. 
\begin{lemma}\label{le:orderEV}
	Let $ W_n $ denote the number of isolated vertices in the preferential attachment graph $ \cG_n $ described in Section~\ref{sec:model}. 
 For any attachment function $ f $ with $f(k)\leq \gamma k +1$ for some $ \gamma \in (0,1)$ and $ f(0)<1 $, we then have that
		\begin{align}\label{eq:mun}
		 \E\left[W_n\right]  \asymp  n 
		\end{align}
			and   $$\Var\left[W_n\right]\, = \mu_n \big(\E[R_{n,I}+D_{n,I}+\1\{\textnormal{\texttt{D}}_n(I)>0\}]\big) \geq c_v\, \mu_n$$
			for some constant $ c_v>0 $ independent of $ n $.	
\end{lemma}
	\begin{lemma}\label{le:Dni}
		Let $ D_{n,I} $ denote the number of neighbours of vertex $ I $ with total degree one in $ \cG_n $. For any attachment rule with $f(k)\leq \gamma k +1$ for some $ \gamma \in (0,1)$ and $ f(0)<1 $, we then have that
		\begin{align*}
		\E[D_{n,I}] \lesssim &\  1 \quad \text{ and } \quad \E[D_{n,I}^2] \lesssim\  \begin{cases}
		1 &\text{  for } \gamma \leq \frac{1}{2},\\
		n^{2\gamma-1}&\text{  for } \gamma > \frac{1}{2}.\\
		\end{cases}
		\end{align*}
	
		Furthermore,
		\begin{align*}
		\sum_{i=1}^n \sum_{j=1}^{i-1}	\Cov(D_{n,i},D_{n,j})\  \lesssim \,	n
		\end{align*}
	\end{lemma}
\begin{lemma}\label{le:Rni}
	Denote by  $ R_{n,I} $ the number of isolated  vertices in $ \cG_{n}^{(I)} $ which are neither isolated in $ \cG_n $ nor contained in $ \mathcal{D}_{n,I} $. For any attachment rule with $f(k)\leq \gamma k +1$ for some $ \gamma \in (0,1)$ and $ f(0)<1 $, we have 
	\begin{align}\label{MomentsRni}
	\E[R_{n,I}]&\ \lesssim  \ 
	\begin{cases}
	1 &\text{ for }\gamma <\frac{1}{2},\\
	\log(n) &\text{ for }\gamma =\frac{1}{2},\\
	n^{2 \gamma-1}&\text{ for }\gamma > \frac{1}{2},\\
	\end{cases}	
	\end{align}
	and
	\begin{align}\label{2ndmomentRni}
	 \E[R_{n,I}^2]&\ \lesssim\
	 \begin{cases}
	1 &\text{ for }\gamma <\frac{1}{3},\\
	\log(n) &\text{ for }\gamma =\frac{1}{3},\\
	n^{3\gamma-1}&\text{ for }\gamma > \frac{1}{3}.\\
	\end{cases}
\end{align}	
	Furthermore,
\begin{align*}
\sum_{i=1}^n \sum_{j=1}^{i-1}	\Cov(R_{n,i},R_{n,j})
&\lesssim \begin{cases}
n &\text{ for } \gamma <\frac{1}{4},\\
n \log(n)&\text{ for } \gamma =\frac{1}{4},\\
n^{4\gamma} &\text{ for } \gamma >\frac{1}{4}.
\end{cases}
\end{align*}
\end{lemma}
The lemmas above are used to derive the following bounds on the two terms appearing in \eqref{eq:termstobound}.
\begin{lemma}\label{le:variancebound}
	For $ W_n^s $ having the size-bias distribution of $ W_n $, there exists a constant $ C>0 $, independent of $ n $, such that
	\begin{equation*}
	\Var[\E\left[W_n^s-W_n\vert \cG_n\right]] \leq \left(\frac{2\sigma_n}{\mu_n}\right)^2+\frac{C}{\mu_n^2}~\begin{cases}
	n&\mbox{ for } \gamma<\frac{1}{4},\\
	n \log(n)&\mbox{ for } \gamma=\frac{1}{4},\\
	n^{4\gamma}&\mbox{ for } \gamma>\frac{1}{4}.\\
	\end{cases}
	\end{equation*}
	
\end{lemma}

\begin{lemma}\label{le:squareddistance}
	For $ W_n $ denoting the number of isolated vertices in a preferential attachment graph $ \cG_n $ described in Section~\ref{sec:model}  and $ W_n^s $ having the size-bias distribution of $ W_n $, there exists a constant $ C $ independent of $ n $ such that
	\begin{equation*}
	\E\left[(W_n^s-W_n)^2\right]\leq C \begin{cases}
	1 &\text{ for }\gamma <\frac{1}{3}\\
	\log(n) &\text{ for }\gamma =\frac{1}{3}\\
	n^{3\gamma-1}&\text{ for }\gamma >\frac{1}{3}\\
	\end{cases}
	\end{equation*}
	
\end{lemma}

With these auxiliary results we are finally ready to prove our main result Theorem~\ref{Thm:isolatedvertices}.
\begin{proof}[Proof of Theorem \ref{Thm:isolatedvertices}]
Plugging the bounds  given in Lemmas~\ref{le:orderEV}, \ref{le:variancebound} and \ref{le:squareddistance} into Equation \eqref{Thmbound} proves the claim.
\end{proof}
\section{Proofs of auxiliary Lemmas}\label{sec:proofs}
\subsection{Proof of Lemma~\ref{le:orderEV}}
\begin{proof}
	Due to independence structure \texttt{\textbf{Ib}} we have
	\begin{align}\label{eq:Expectation}
	\E[W_n]=\E[\sum_{i=1}^n X_i]=\sum_{i=1}^n \P(\ind_n(i)=0)\ \P(\out(i)=0).
	\end{align}
	Now, for every $ i \in [n] $,
	\begin{align}\label{eq:Variance}
	\P(\ind_n(i)=0)= \prod_{\ell=i+1}^n\Big(1-\frac{f(0)}{\ell-1}\Big) \asymp \big(\frac{i}{n}\big)^\eta
	\end{align}
	where $ \eta:= f(0) $.
	According to \cite[Theorem 1.6]{BDO18}, the outdegree asymptotically follows a Poisson distribution with parameter $ \lambda \in (0, \infty)$. Hence for every $\varepsilon> 0$ there exists $ N=N(\varepsilon) \in \N  $ such that for all $ n \geq N  $
	\[
	\P(\out(n)=0)\geq (1-\varepsilon) e^{-\lambda},
	\]	
	in particular, there exists $N^\ast$ such that for all $n \geq N^\ast$
	\[
	\P(\out(n)=0)\geq\frac{1}{2}e^{-\lambda},
	\]	
	Furthermore, for any fixed $ k \in \N $ we have
	\[
	p_{0,k}:=\P(\out(k)=0)= \prod_{\ell=1}^{k-1}\Big(1-\frac{\E[\ind_{k-1}(\ell)]}{k-1}\Big) > 0.
	\]
	Thus, for all $ k \in \N_{\geq 2} $
		\begin{align}\label{eq:outdegree0}
	p_{0,k}=\P(\out(k)=0)\geq \min\{p_{0,1}, \ldots, p_{0,N^\ast}, \frac{1}{2} e^{-\lambda} \}=:p_0  >0
	\end{align}
	and \eqref{eq:mun} follows by \eqref{integraltest}.
	We now turn to the variance bound. By \eqref{eq.SizeBiasAlternative} we see that
	\begin{align*}
	\E[W_n^s]&=\frac{1}{\E[W_n]}\E[W_n^2],	
	\end{align*}
	so that
	\begin{align*}
	\Var[W_n]&=\E[W_n^2]-\E[W_n]^2=\E[W_n]\,\E[W_n^s-W_n]\\
	&=\E[W_n]\left(\E[\1\{\D_n(I)>0\}+D_{n,I}+R_{n,I}]\right)\geq c_{v}\, n,
	\end{align*}
	since 
	$$
	\E[\1\{\textnormal{\texttt{D}}_n(I)>0\}]\geq \frac{1}{\mu_n} \sum_{i=1}^n \vartheta_{i,n} \P(\out(i)>0)  \geq \frac{1}{\mu_n}\sum_{i=2}^n  \vartheta_{i,n} (1-p_0)> (1-p_0)\Big(1-\frac{1}{\mu_n}\Big).
	$$
	
\end{proof}
\subsection{Proof of Lemma~\ref{le:Dni}}
For the proof of Lemma~\ref{le:Dni} we need the following Proposition, which gives an upper bound on the impact of an isolated vertex on the outdegrees of vertices in the network.
\begin{prop}\label{Prop:ajl}
	For any $ j,i\in [n]$, $ J \subset \{1, \dots ,\min\{i,j\} \}$ and any attachment rule $ f  $ with $f(k)\leq \gamma k +1$ for some $ \gamma \in (0,1)$ and $ f(0)<1 $ we have
	\begin{align*}
	\prod_{k\in J} \P((j,k)\notin \cG\vert (i,k)\notin \cG)- \prod_{k\in J} \P((j,k)\notin \cG)\lesssim (ij)^{\gamma-1}\sum_{k=1}^\ell k^{-2\gamma}=: \xi^\ell_{j,i}, 
	\end{align*}
	where $ \ell $ denotes the largest integer in $ J$.
\end{prop}
\begin{proof}
	For any $ k\in J $ we have 
	\begin{align*}
	&\P((j,k)\notin \cG\vert (i,k)\notin \cG)\\
	&=\P((j,k)\notin \cG)+(\P((j,k)\notin \cG\vert (i,k)\notin \cG)-\P((j,k)\notin \cG\vert (i,k)\in \cG))\P((i,k)\in \cG)\\
	&=\P((j,k)\notin \cG)+(\P((j,k)\in \cG\vert (i,k)\in \cG)-\P((j,k)\in \cG\vert (i,k)\notin \cG))\P((i,k)\in \cG),
	\end{align*}
	so that for any $ \ell \in J $
	\begin{align*}
	&\prod_{k\in J}\P((j,k)\notin \cG\vert (i,k)\notin \cG)\leq \P((j,\ell)\notin \cG)\cdot \prod_{k\in J \backslash \{\ell\}}\P((j,k)\notin \cG\vert (i,k)\notin \cG) \\
	&\qquad \qquad \qquad \qquad \qquad \qquad + (\P((j,\ell)\in \cG\vert (i,\ell)\in \cG)-\P((j,\ell)\in \cG\vert (i,\ell)\notin \cG))\P((i,\ell)\in \cG).
	\end{align*}
	By iteration we thus 
	\begin{align*}
	&\prod_{k\in J}\P((j,k)\notin \cG\vert (i,k)\notin \cG)\\
	&\quad \leq\prod_{k\in J \backslash \{\ell\}}\P((j,k)\notin \cG)+ \sum_{k\in J}(\P((j,k)\in \cG\vert (i,k)\in \cG)-\P((j,k)\in \cG\vert (i,k)\notin \cG))\P((i,k)\in \cG).
	\end{align*}
 Lemma~\ref{le:DM 2.8}  in combination with Lemma~\ref{le:StochasticDomination} shows that
	\begin{align*}
	(\P((j,k)\in \cG\vert (i,k)\in \cG)-\P((j,k)\in &\cG\vert (i,k)\notin \cG))\P((i,k)\in \cG)\leq \frac{f(1)}{f(0)}\P((j,k)\in \cG)\,\P((i,k)\in \cG)
	\end{align*}
	so that Lemma~\ref{le:moment} yields the result.
\end{proof}
\begin{proof}[Proof of Lemma~\ref{le:Dni}]We start by introducing the family of random variables
	\begin{align}\label{def:Y}
	Y_{n,j}^{(i)}&:=\1\{\text{vertex $ i $ is the only neighbour of vertex  }j\text{ in } \cG_n\},
	\end{align}
	so that $ D_{n,i}=\sum_{j=1}^n Y_{n,j}^{(i)} $. Note that for every $ j \in [n] $ the random variable $ Y_{n,j}^{(i)} $ can be one for at most one $ i \in [n] $. Hence,  $ \sum_{i=1}^n D_{n,i}\leq n $ and thus
	\begin{align*}
	\E[D_{n,I}] =\frac{1}{\mu_n} \E[\sum_{i=1}^n \vartheta_{i,n}D_{n,i}]\lesssim 1
	\end{align*}
by Lemma~\ref{le:orderEV}.	We now turn to the second-order properties of $ D_{n,i} $ for $ i \in [n] $. 
We have
	\begin{align}\label{Eq:2ndMomentDni}
\E[D_{n,i}^2]=\E[D_{n,i}]+ 2 \sum_{\ell=1}^n \sum_{m=1}^{\ell-1}\P(Y_{n,\ell}^{(i)}=1)\P(Y_{n,m}^{(i)}=1\vert Y_{n,\ell}^{(j)}=1).
\end{align}
and for the  the covariances
	\begin{align} \label{CovDniDnj}
	\Cov(D_{n,i},D_{n,j})&=\sum_{\ell=1}^n \sum_{m=1}^n\Big( \P(Y_{n,\ell}^{(i)}Y_{n,m}^{(j)}=1)-\P(Y_{n,\ell}^{(i)}=1)\P(Y_{n,m}^{(j)}=1)\Big)\notag\\
	&=\sum_{\ell=1}^n \sum_{m=1}^n \P(Y_{n,\ell}^{(i)}=1)\Big(\P(Y_{n,m}^{(j)}=1\vert Y_{n,\ell}^{(i)}=1)-\P(Y_{n,m}^{(j)}=1)\Big).
	\end{align}
To deal with these expressions we have to consider the conditional probabilities $ \P(Y_{n,m}^{(j)}=1\vert Y_{n,\ell}^{(i)}=1) $. 
First note that 
\[
\P((2,1)\notin \cG)= 1- f(0)=1-\eta >0
\]
and
\begin{align}\label{eq:ProbEdgeNotInG}
\P((m,\ell)\notin \cG)&= 1- \frac{\E[f(\deg_{m-1}(\ell))]}{m-1} \geq 1-\frac{\gamma(m-1-\ell)+1}{m-1}\notag\\&=1-\gamma+ \frac{\gamma \ell-1}{m-1}\geq \frac{1}{2}(1-\gamma)
\end{align}
for $ m\geq 3 $ and all $ \ell \in [m-1]. $ Denote by $ c_{\eta, \gamma}=\min\{1-\eta, \frac{1}{2}(1-\gamma)\}^{-1} $. Then,
	\begin{align}\label{ineq:inverseredge}
	\P((m,\ell)\notin \cG)^{-1}=1+ \frac{\P((m,\ell)\in \cG)}{\P((m,\ell) \notin \cG)} \leq 1+c_{\eta, \gamma}\,  \P((m,\ell)\in \cG). 
	\end{align}
	
 We now distinguish the possible cases of constellations of $ \ell,m, i $ and $ j $, with $ i>j $, to deal with the conditional probabilities $ \P(Y_{n,m}^{(j)}=1\vert Y_{n,\ell}^{(i)}=1)  $.
	For $ \ell<m<j $ Proposition~\ref{Prop:ajl}, independence structures \texttt{\textbf{Ia}},\texttt{\textbf{ Ib}} and the inequality in \eqref{ineq:inverseredge} yield 
	\begin{align}\label{DniCov1}
	&\P(Y_{n,m}^{(j)}=1\vert Y_{n,\ell}^{(i)}=1)\notag\\
	&\qquad =\P(\o(m)=0, (j,m)\in \cG,\d_{n}(m)=1\vert \o(\ell)=0, (i,\ell)\in \cG,\d_{n}(\ell)=1))\notag\\
	&\qquad = \ \prod_{r=1}^{\ell-1}\P((m,r)\notin \cG\vert (\ell,r) \notin \cG) \cdot \prod_{r=\ell+1}^{m-1}\P((m,r)\notin \cG) \cdot \P(\d_n(m)=1, (j,m) \in \cG)\notag\\
		&\qquad\leq  \Big( \prod_{r=1}^{\ell-1}\P((m,r)\notin \cG)+\xi_{m,\ell}^{\ell-1}\Big) \cdot \prod_{r=\ell+1}^{m-1}\P((m,r)\notin \cG)  \P(\d_n(m)=1, (j,m) \in \cG)\notag\\
	&\qquad\leq\P(Y_{n,m}^{(j)}=1)\big(1+c_{\eta, \gamma}\, \P((m,\ell)\in \cG)\big)\notag\\
	&\qquad \qquad+ \xi_{m,\ell}^{\ell-1} \prod_{r=\ell+1}^{m-1}\P((m,r)\notin \cG) \cdot \P(\d_n(m)=1, (j,m) \in \cG)\notag\\
	&\qquad \leq \P(Y_{n,m}^{(j)}=1)\Big(1+c_{\eta, \gamma}\, \P((m,\ell)\in \cG)+p_0^{-1}\xi_{m,\ell}^{\ell-1}\Big), 
	\end{align}	
	since
	\[
	\prod_{r=1}^{m-1}\P((m,r)\notin \cG)=\P(\out(m)=0)\geq p_0.	\] 
	Analogously, we obtain for $ m<\ell $ and $ m < j $
	\begin{align}\label{DniCov2}
	&\P(Y_{n,m}^{(j)}=1\vert Y_{n,\ell}^{(i)}=1)\notag\\
	&\qquad = \ \prod_{r=1}^{m-1}\P((m,r)\notin \cG\vert (\ell,r) \notin \cG)  \P(\d_n(m)=1, (j,m) \in \cG\vert (\ell,m) \notin \cG)\notag\\
	&\qquad\leq\P(Y_{n,m}^{(j)}=1)\Big(1+c_{\eta, \gamma}\, \P((\ell,m)\in \cG)+c^\ast \xi_{m,\ell}^{m-1}\Big),
	\end{align}
	where $ c^\ast= c_{\eta,\gamma}\,  p_0^{-1} $.
	Note that due to independence structure \texttt{\textbf{Ia}} all calculations up to this point  hold irrespective of whether $ i=j $ or $ i\neq j $. However, this is no longer true for $ m>  j$ since in this case the events $ \{Y_{n,m}^{(j)}=1\} $ and $ \{Y_{n,\ell}^{(j)}=1\} $ both depend on the indegree of vertex $ j$. For $ i \neq j $ and $ j<m<\ell $ we get
	\begin{align}\label{DniCov3}
	&\P(Y_{n,m}^{(j)}=1\vert Y_{n,\ell}^{(i)}=1)\notag\\
	&\quad = \ \prod_{r=1, r\neq j}^{m-1}\P((m,r)\notin \cG\vert (\ell,r) \notin \cG) \ \P((m,j) \in \cG\vert (\ell,j) \notin \cG) \P(\d_n(m)=0\vert (\ell,m) \notin \cG)\notag\\
	&\quad\leq\P(Y_{n,m}^{(j)}=1)\Big(1+c_{\eta,\gamma}\,\P((\ell,m)\in \cG)\big)+ c^\ast\, \xi_{m,\ell}^{m-1}\Big),
	\end{align}
	since $ \P((m,j) \in \cG\vert (\ell,j) \notin \cG)\leq \P((m,j) \in \cG) $. \\
	For $ i \neq j $ the last case to consider is the case $ m> j, \ell.  $ We have
	\begin{align}\label{DniCov4}
	&\P(Y_{n,m}^{(j)}=1\vert Y_{n,\ell}^{(i)}=1)\notag\\
	&\quad \leq\ \prod_{r=1, r \neq j}^{\ell-1}\P((m,r)\notin \cG\vert (\ell,r) \notin \cG) \prod_{r=\ell+1, r\neq j}^{m-1}\P((m,r)\notin \cG)\ \P((m,j) \in \cG) \P(\d_n(m)=0)\notag\\
	&\quad\leq\P(Y_{n,m}^{(j)}=1)\Big(1+c_{\eta,\gamma}\,\P((m,\ell)\in \cG)+p_0^{-1}\,\xi_{m,\ell}^{\ell-1}\Big).
	\end{align}
	For $ i=j $, we need to replace  $ \P((m,j) \in \cG) $ with $ \P((m,j) \in \cG\vert (\ell,j) \in \cG) $ for $ m,\ell >j $, i.e. in the last two cases. Since 
	$$ \P((m,j) \in \cG\vert (\ell,j) \in \cG) \leq\frac{f(1)}{f(0)}\P((m,j) \in \cG) 
	$$  according to Lemma~\ref{le:DM 2.8},
	we obtain
	\begin{align*}
	&\P(Y_{n,m}^{(j)}=1\vert Y_{n,\ell}^{(j)}=1)
	\ \lesssim \ \P(Y_{n,m}^{(j)}=1)
	\end{align*}
	for any $ m,\ell, j $. Note that for $ m> j $ we have
	\begin{align*}
	\P(Y_{n,m}^{(j)}=1)\,\leq \, 	\P(\d_n(m)=0) \P((m,j)\in \cG)\, \lesssim \, \Big(\frac{m}{n}\Big)^\eta m^{\gamma-1}	j^{-\gamma} 
	\end{align*}
	and 
		\begin{align*}
	\P(Y_{n,m}^{(j)}=1)\,&\leq \,	\prod_{r=m+1}^{j-1}\P((r,m)\notin \cG\vert \d_{r-1}(m)=0) \P((j,m)\in \cG\vert \d_{j-1}(m)=0)\\
	&\qquad \qquad \times 	\prod_{r=j+1}^{n}\P((r,m)\notin \cG\vert \d_{r-1}(m)=1) \\
	&\lesssim\,  \Big(\frac{m}{n}\Big)^\eta j^{-1} 
	\end{align*}
	for $ m<j $. Combing Equation~\eqref{Eq:2ndMomentDni}  with the considerations above and  finally using \eqref{integraltest} we thus obtain
	\begin{align}\label{DniVar}
\E[D_{n,I}^2]&=	\frac{1}{\mu_n}\sum_{j=1}^n \vartheta_{j,n}\E[D_{n,j}^2]\lesssim \  \frac{1}{\mu_n}\sum_{j=1}^n \vartheta_{j,n}\Big(\E[D_{n,j}]+2\sum_{\ell=1}^n \sum_{m=1}^{\ell-1}\P(Y_{n,\ell}^{(j)}=1)\P(Y_{n,m}^{(j)}=1)	\Big)\notag\\
&\lesssim    \frac{1}{\mu_n}\sum_{j=1}^n \vartheta_{j,n}\Big(\E[D_{n,j}]+2\sum_{\ell=1}^{j-1}\sum_{m=1}^{\ell-1}\P(Y_{n,\ell}^{(j)}=1)\P(Y_{n,m}^{(j)}=1)\notag\\
&\qquad  +2	\sum_{\ell=j+1}^{n}\Big(\sum_{m=1}^{j-1}\P(Y_{n,\ell}^{(j)}=1)\P(Y_{n,m}^{(j)}=1) + \sum_{m=j+1}^{\ell-1}\P(Y_{n,\ell}^{(j)}=1)\P(Y_{n,m}^{(j)}=1)\Big)\Big)\notag\\
&\lesssim 
\begin{cases}
1 &\text{  for } \gamma \leq \frac{1}{2},\\
n^{2\gamma-1}&\text{  for } \gamma > \frac{1}{2}.
\end{cases}
	\end{align}
	For the covariance we first remark that
	\[
	\P((m,\ell)\in \cG)\lesssim \xi_{m,\ell}^{\ell} \qquad \text{ and } \qquad \P((\ell,m)\in \cG)\lesssim \xi_{m,\ell}^{m}
	\] so that by plugging
	\eqref{DniCov1}, \eqref{DniCov2}, \eqref{DniCov3} and \eqref{DniCov4} into \eqref{CovDniDnj} we  obtain
	\begin{align}\label{DniCov}
	&\Cov(Y_{n,\ell}^{(i)},Y_{n,m}^{(j)})\, \lesssim \, \P(Y_{n,\ell}^{(i)}=1)\,\P(Y_{n,m}^{(j)}=1)\, \xi_{m,\ell}^{m\wedge \ell},
	\end{align}
	where $ m\wedge \ell=\min\{m,\ell\} $.
    Keeping in mind that $ Y_{n,m}^{(i)}Y_{n,m}^{(j)}=0 $ for $ i \neq j $, we finally obtain
	\begin{align*}
	&\sum_{i=1}^n \sum_{j=1}^{i-1}\Cov(D_{n,i},D_{n,j}) =\sum_{i=1}^n \sum_{j=1}^{i-1} \sum_{m=1}^n\sum_{\ell=1}^{n}   \Cov(Y_{n,\ell}^{(i)},Y_{n,m}^{(j)})\\
	&\leq \sum_{i=1}^n \sum_{j=1}^{i-1} \Bigg[\sum_{m=1}^{j-1}j^{-1}\Bigg(\sum_{\ell=1}^{m-1}i^{-1} m^{\gamma-1} \ell^{-\gamma} +\sum_{\ell=m-1}^{i-1}i^{-1}m^{-\gamma} \ell^{\gamma-1} +\sum_{\ell=i+1}^n(mi)^{-\gamma}\ell^{2\gamma-2}\Big(\frac{\ell}{n}\Big)^\eta\Bigg)\\ 
	&\qquad \qquad +\sum_{m=j+1}^{i-1}m^{\gamma-1}j^{-\gamma}\Big(\frac{m}{n}\Big)^\eta\Bigg(\sum_{\ell=1}^{m-1} i^{-1}
	 m^{\gamma-1}\ell^{-\gamma}+\sum_{\ell=m+1}^{i-1}i^{-1}m^{\gamma-1}\ell^{-\gamma}  \\ 
	&\qquad  \qquad+\sum_{\ell=i+1}^{n}(im)^{-\gamma}\ell^{2\gamma-2} \Big(\frac{\ell}{n}\Big)^\eta\Bigg)+\sum_{m=i+1}^{n}m^{\gamma-1}j^{-\gamma}\Big(\frac{m}{n}\Big)^\eta \Bigg(\sum_{\ell=1}^{i-1}i^{-1}m^{\gamma-1}\ell^{-\gamma}\\
	&\qquad  \qquad +\sum_{\ell=i+1}^{m-1}\ell^{-1} i^{-\gamma}m^{\gamma-1}\Big(\frac{\ell}{n}\Big)^\eta+\sum_{\ell=m+1}^{n}(im)^{-\gamma}\ell^{2\gamma-2} \Big(\frac{\ell}{n}\Big)^\eta \Bigg)\Bigg]
	\lesssim n,	
	\end{align*}
	where we made repeated use of \eqref{integraltest}.
\end{proof}
\subsection{Proof of Lemma~\ref{le:Rni}}
	\begin{proof}
		First of all note that edges with both endpoints younger than vertex $ i $ are not affected by the isolation of vertex $ i $, i.e they are present in $\bar{\cG}^{(i)}$ if and only if they are present in $ \cG $. Furthermore, every edge not present in $ \bar{\cG}^{(i)} $ that is part of $ \cG $ can produce at most two additional isolated vertices. Thus, using Corollary~\ref{le:indegreediff}, we see that
		\begin{align*}
		 \E[R_{n,I}]= \frac{1}{\mu_n}\sum_{i=1}^n \vartheta_{i,n} \E[R_{n,i}] &\leq \frac{2}{\mu_n}\sum_{i=1}^n \sum_{k=1}^{i-1}\E[\Delta_n^{(i)}\texttt{d}(k)]  \leq \frac{2}{\mu_n} \sum_{i=1}^n \sum_{k=1}^{i-1} n^\gamma i^{\gamma-1}k^{-2\gamma} \\
		 &\asymp
		\begin{cases}
		1 &\text{ for } \gamma < \frac{1}{2},\\
		\log(n) &\text{ for } \gamma = \frac{1}{2},\\
		n^{2\gamma-1} &\text{ for } \gamma> \frac{1}{2}.
		\end{cases}
		\end{align*}
		Turning now to the second moment of $ R_{n,I} $, we have
		\begin{align}\label{eq:Rni}
		\E[R_{n,I}^2]&= \frac{1}{\mu_n}\sum_{i=1}^n \vartheta_{i,n}\E[R_{n,i}^2]\notag\\
		&\leq \frac{4}{\mu_n}\E\Big[\sum_{i=1}^n \Big(\sum_{k=1}^{i-1} \Delta^{(i)}_{n} \texttt{d}(k)^2+  \sum_{k=1}^{i-1}\sum_{\ell=1, \ell \neq k}^{i-1} \Delta^{(i)}_{n} \texttt{d}(k) \Delta^{(i)}_{n} \texttt{d}(\ell)\Big)\Big]. 
		\end{align}
		By the construction of the random graphs $ \cG $ and $ \bar{\cG}^{(i)} $ the random variables $ \Delta^{(i)}_{n} \texttt{d}(k) $ and $ \Delta^{(i)}_{n} \texttt{d}(\ell) $ are independent for $ k \neq \ell $ (cf. independence structure \texttt{\textbf{Ia}}), so that
		\begin{align}\label{Rnimixed}
		\E\Big[\sum_{i=1}^n\sum_{k=1}^{i-1}\sum_{\ell=1, \ell \neq k}^{i-1} \Delta^{(i)}_{n} \texttt{d}(k) \Delta^{(i)}_{n} \texttt{d}(\ell)\Big]&= \sum_{i=1}^n\sum_{k=1}^{i-1}\sum_{\ell=1, \ell \neq k}^{i-1} \E\big[\Delta^{(i)}_{n} \texttt{d}(k)\big] \E\big[\Delta^{(i)}_{n} \texttt{d}(\ell)\big]\notag\\
		&\lesssim \begin{cases}
		n&\text{ for } \gamma < \frac{1}{2},\\
		n\log(n)^3 &\text{ for } \gamma = \frac{1}{2},\\
		n^{4\gamma-1} &\text{ for } \gamma> \frac{1}{2},
		\end{cases}
		\end{align}
		according to Corollary~\ref{le:indegreediff}. 
		To deal with the first term in \eqref{eq:Rni} note that
		\begin{align*}
		&\E\big[\Delta^{(i)}_{n} \texttt{d}(k)^2\big]= \E\Big[\sum_{r=k+1}^n \1\{(r,k)\in \cG, (r,k)\notin \cGi\} \\
		&\qquad \qquad \qquad \qquad \qquad + \sum_{r=k+1}^n\sum_{s=k+1\atop s \neq r}^n \1\{(r,k)\in \cG, (r,k)\notin \cGi\}\1\{(s,k)\in \cG, (s,k)\notin \cGi\}\Big]\\
		&\qquad \leq \E\big[\Delta^{(i)}_{n} \texttt{d}(k)\big]+ \sum_{r=k+1}^n\sum_{s=k+1\atop s \neq r}^n \P((r,k)\in \cG,(s,k)\in \cG, (i,k)\in \cG)\\
		&\qquad \lesssim n^\gamma i^{\gamma-1}k^{-2\gamma}+n^{2\gamma}i^{\gamma-1} k^{-3\gamma}\, \asymp \, n^{2\gamma}i^{\gamma-1} k^{-3\gamma}
		\end{align*}
		and thus
		\begin{align}\label{SecondmomentDelta}
		\sum_{i=1}^n\sum_{k=1}^{i-1}\E\big[\Delta^{(i)}_{n} \texttt{d}(k)^2\big]&\lesssim   \begin{cases}
		n &\text{ for }\gamma <\frac{1}{3},\\
		n\log(n) &\text{ for }\gamma =\frac{1}{3},\\
		n^{3\gamma}&\text{ for }\gamma >\frac{1}{3}.\\
		\end{cases}
		\end{align}
				
		Plugging \eqref{Rnimixed} and \eqref{SecondmomentDelta} into \eqref{eq:Rni} yields \eqref{2ndmomentRni}.\\
		We now turn to the covariances. For $ j <i  $ we have
			\begin{align*}
			\Cov(R_{n,i},R_{n,j})&=\E\big[\sum_{\ell=1}^n\sum_{m=1}^n \1\{\ell \in \mathsf{R}_{n,i}\} \1\{m \in \mathsf{R}_{n,j}\}\big]\\
			& \qquad \qquad -\E\big[\sum_{\ell=1}^n\1\{\ell \in \mathsf{R}_{n,i}\}\big]\ \E\big[\sum_{m=1}^n  \1\{m \in \mathsf{R}_{n,j}\}\big].
			\end{align*}
				To deal with this differences, we will condition on the event that vertices $ \ell $, $ m $, $  j  $ and $ i $  have a common older neighbour in $ \cG $. 
				To do so, we denote by $ \textbf{d}_\mathcal{G} $ the geodesic graph distance of vertices in $ \cG $, i.e. for vertices $ m $ and $ \ell $ in $ \cG $ $ \textbf{d}_\mathcal{G}(m, \ell) $ denotes the minimal number of edges in a path connecting vertices $ m $ and $ \ell $. If there is no path connecting the two vertices we put  $ \textbf{d}_\mathcal{G}(m , \ell)=\infty. $
				We then define 
				\begin{align*}
				&\mathsf{N}^{(i,j)}(\ell,m)=\{ k \in [\min\{j,m,\ell\}]: \textbf{d}_\mathcal{G}(m, k)\leq 1,\, \textbf{d}_\mathcal{G}( \ell,k)\leq 1,\, \textbf{d}_\mathcal{G}(j,k) \leq 1,\, \textbf{d}_\mathcal{G}(i,k)\leq 1\}.
				\end{align*}
				Due to the construction of the coupled graph $\bar{\cG}^{(i)}$  				
				the event $ \{\ell \in \mathsf{R}_{n,i}\} $ depends on the existence or non-existence of edges in the following two sets:
				\begin{align*}
				\mathsf{E}^{(i)}_1(\ell)&= \{(r,k)\in \cG: k<\min\{\ell,i\}, r\geq k+1 \text{ with }  (\ell,k) \in \cG \text{ and }  (i,k) \in \cG \}. 
				\end{align*}
				and for $ i > \ell $ with  ($ i,\ell) \in \cG  $
				\begin{align*}
				\mathsf{E}^{(i)}_2(\ell)&= \{(r,\ell)  \in \cG: r \geq \ell+1\}.
				\end{align*}
				Note that $ \mathsf{E}^{(i)}_2(\ell)\cap \mathsf{E}^{(j)}_2(m)=\emptyset $ for pairwise distinct $  i,j,\ell, m \in \N $ and the existence of edges in $ \mathsf{E}^{(i)}_2(\ell) $ is independent of the existence of edges in $  \mathsf{E}^{(j)}_2(m) $ for $ m\neq \ell $ (see independence structure \texttt{Ia}). Moreover, 
				\[
				\mathsf{E}_1^{(i)} (\ell) \,\cap \,\mathsf{E}_1^{(i)} (m)= \{(r,k) \in \cG: k< \min\{i,j,\ell,m \},\{ (i,k), (j,k)\, (\ell,k), (m,k)\}\in \cG\},
				\]
				so that 			
				on the event $ \{\mathsf{N}^{(i,j)}(\ell,m)= \emptyset  \}$ we also have $ \mathsf{E}_1^{(i)} (\ell) \cap \mathsf{E}_1^{(i)} (m)=\emptyset$, which means that the events $ \{\ell \in \mathsf{R}_{n,i}\} $  and $ \{m \in \mathsf{R}_{n,j}\} $ depend on disjoint and independent sets of edges. 
			   Hence,
			   \begin{align*}
			   &\E\big[ \1\{\ell \in \mathsf{R}_{n,i}\} \1\{m \in \mathsf{R}_{n,j}\}\1\{\mathsf{N}^{(i,j)}(\ell,m)= \emptyset \}\big] -\E\big[\1\{\ell \in \mathsf{R}_{n,i}\}\big]\ \E\big[\  \1\{m \in \mathsf{R}_{n,j}\}\big]\\
			   &\leq \P(\ell \in \mathsf{R}_{n,i})\Big(\P(m \in \mathsf{R}_{n,j}\vert \mathsf{N}^{(i,j)}(\ell,m)= \emptyset )- \P(m \in \mathsf{R}_{n,j})\Big)\\
			   &\leq   \P(\ell \in \mathsf{R}_{n,i})\P(m \in \mathsf{R}_{n,j})\ \frac{\P(\mathsf{N}^{(i,j)}(\ell,m)\neq \emptyset)}{\P(\mathsf{N}^{(i,j)}(\ell,m)= \emptyset)}\lesssim \P(\mathsf{N}^{(i,j)}(\ell,m)\neq \emptyset),
			   \end{align*}
			   as 
			   \[
			   \P(\mathsf{N}^{(i,j)}(\ell,m)= \emptyset)\geq p_0.
			   \]
			  			   To deal with   $ \P(\mathsf{N}^{(i,j)}(\ell,m)\neq \emptyset) $, let $ \mathsf{k}_{j,m,\ell}^\ast$ denote the oldest vertex (i.e. the vertex with the smallest label) in $ \mathsf{N}^{(i,j)}(\ell,m) $.  We then have
			  			   \begin{align*}
			  			   \P(\mathsf{N}^{(i,j)}(\ell,m)\neq \emptyset)= \sum_{k=1}^{\min \{j,m,\ell\}-1} \P(\mathsf{k}_{j,m,\ell}^\ast=k)+ \P(\mathsf{k}_{j,m,\ell}^\ast=\min \{j,m,\ell\} ).
			  			   \end{align*}
			  			For $ k \in \{1,\ldots,\min \{j,m,\ell\}-1\} $ we have
			  			   \begin{align*}
			  			 \P(\mathsf{k}_{j,m,\ell}^\ast=k) &\leq  \P((\ell,k)\in \cG, (m,k)\in \cG, (i,k) \in \cG, (j,k) \in \cG)\lesssim (ij\ell m)^{\gamma-1} k^{-4\gamma}.
			  			   \end{align*}
			  			   Furthermore, for an ordering $ \{r_1,r_2, r_3 \}$ of $ \{j,m,\ell\} $ with $  r_1<r_2<r_3$   we have
			  			   \[
			  			   \P(\mathsf{k}_{j,m,\ell}^\ast=\min \{j,m,\ell\} )\lesssim r_1 ^{-3\gamma } (r_2 r_3i)^{\gamma-1}.
			  			   \]
			  			Straightforward case distinctions in combination with repeated use of  \eqref{integraltest} lead to
			   \begin{align*}
			  \sum_{i=1}^n\sum_{j=1}^{i-1} \sum_{m=1}^n \sum_{\ell=1}^n \P(\mathsf{k}_{j,m,\ell}^\ast=\min \{j,m,\ell\} )\lesssim
			  &\begin{cases}
			  n&\text{ for } \gamma < \frac{1}{3},\\
			 n\log(n)   &\text{ for } \gamma = \frac{1}{3},\\
			     n^{3\gamma} &\text{ for } \gamma > \frac{1}{3}.\\
			  \end{cases}
			   \end{align*}
				Consequently,
			\begin{align*}
			&	\sum_{i=1}^n \sum_{j=1}^{i-1}	\Cov(R_{n,i},R_{n,j})\\
		&=\sum_{i=1}^n\sum_{j=1}^{i-1}\sum_{\ell=1}^n\sum_{m=1}^n\Big(\E[\1\{\ell \in \mathsf{R}_{n,i}\} \1\{m \in \mathsf{R}_{n,j}\} ]- \E\big[\1\{\ell \in \mathsf{R}_{n,i}\}\big]\ \E\big[\  \1\{m \in \mathsf{R}_{n,j}\}\big] \Big)\\
	& \lesssim\sum_{i=1}^n\sum_{j=1}^{i-1}\sum_{\ell=1}^n\sum_{m=1}^n\P(\mathsf{N}^{(i,j)}(\ell,m)\neq \emptyset)\\
		&\lesssim\sum_{i=1}^n\sum_{j=1}^{i-1} \sum_{\ell=1}^n\sum_{m=1}^n\Big(\sum_{k=1} ^{\min\{j,\ell,m\}-1}(ij\ell m)^{\gamma-1} k^{-4\gamma}+ \P(\mathsf{k}_{j,m,\ell}^\ast=\min \{j,m,\ell\} )\Big)\\
		&\lesssim
 \begin{cases}
	n &\text{ for } \gamma <\frac{1}{4},\\
	n\log(n)&\text{ for } \gamma =\frac{1}{4},\\
	n^{4\gamma} &\text{ for } \gamma >\frac{1}{4}.
	\end{cases}
		\end{align*}
		\end{proof}

\subsection{Proof of Lemma~\ref{le:variancebound}}
\begin{proof}	
	We have
	\begin{align*}
	&\Var[\E\left[W_n^s-W_n\vert \cG_n\right]]=\frac{1}{\mu_n^2}\Var\Big(\sum_{i=1}^n \vartheta_{i,n}(R_{n,i}+D_{n,i}+\1\{\D_n(i)>0\})\Big)\\	
	&\quad \leq \frac{3}{\mu_n^2}\Big(\Var\big(\sum_{i=1}^n \vartheta_{i,n} R_{n,i}\big)+\Var\big(\sum_{i=1}^n \vartheta_{i,n} D_{n,i}\big)+\Var\big(\sum_{i=1}^n \vartheta_{i,n}\1\{\D_n(i)>0\}\big)\Big).
	\end{align*}
	To bound the last of the three terms note that
	\begin{align*}
	\Var\big(\sum_{i=1}^n \vartheta_{i,n}\1\{\texttt{D}_n(i)>0\}\big)\leq \Var\big(\sum_{i=1}^n \1\{\texttt{D}_n(i)>0\}\big),
	\end{align*}
	since $ \Cov(\1\{\texttt{D}_n(i)>0\},\1\{\texttt{D}_n(j)>0\}) >0 $ for all $ i,j \in [n] $ and $\vartheta_{i,n} \in (0,1) \ \forall i \in [n]$. Furthermore, 
	\begin{align}\label{VarInd}
	\Var\big(\sum_{i=1}^n\1\{\texttt{D}_n(i)>0\}\big)= \Var(n-W_n)=\sigma_n^2.
	\end{align}
	To deal with $ \Var(\sum_{i=1}^n D_{n,i}) $ we use Lemma~\ref{le:Dni} to obtain
	\begin{align}\label{VarDni}
	\Var\big(\sum_{i=1}^n D_{n,i}\big)&=\sum_{i=1}^n\Var( D_{n,i})+ 2\sum_{i=1}^n \sum_{j=1}^{i-1} \Cov(D_{n,i},D_{n,j})\notag \\
	&\leq \mu_n \E[D_{n,I}^2]+ n \asymp\, \begin{cases}
	n &\text{  for } \gamma \leq\frac{1}{2},\\
	n^{2\gamma}&\text{  for } \gamma > \frac{1}{2}.\\
	\end{cases}
	\end{align}
	It remains to deal with $ \Var\big(\sum_{i=1}^n R_{n,i}\big) $. Using Lemma~\ref{le:Rni} we see that
	\begin{align} \label{VarRni}
	\Var\big(\sum_{i=1}^n R_{n,i}\big) &\leq \sum_{i=1}^n\E[ R_{n,i}^2]+2\sum_{i=1}^n \sum_{j=1}^{i-1} \Cov(R_{n,i}R_{n,j})\notag\\
	&= \mu_n \E[R_{n,I}^2]+ 2\sum_{i=1}^n \sum_{j=1}^{i-1} \Cov(R_{n,i}R_{n,j})\notag\\
	&\lesssim  \begin{cases}
	n &\text{ for } \gamma < \frac{1}{4}\\
	n\log(n) &\text{ for } \gamma = \frac{1}{4}\\
	n^{4\gamma} &\text{ for } \gamma > \frac{1}{4}.
	\end{cases}
	\end{align}
	Combining \eqref{VarInd}, \eqref{VarDni} and \eqref{VarRni} proves the claim.

\end{proof}

\subsection{Proof of Lemma~\ref{le:squareddistance}}

\begin{proof}
	We have
	\begin{align*}
	\E\left[(W_n^s-W_n)^2\right]&=\frac{1}{\mu_n} \sum_{i=1}^n \vartheta_{i,n} \E[(W_n^{s,i}-W_n)^2] \\
	&\leq \frac{3}{\mu_n} \sum_{i=1}^n \vartheta_{i,n}\big( \E[R_{n,i}^2]+ \E[D_{n,i}^2]+ \P(\texttt{d}(i)>0)\big)\\
	&\lesssim  \E[R_{n,I}^2]+ \E[D_{n,I}^2]+ 1
	\end{align*}
	which in combination with Lemmas~\ref{le:Dni} and \ref{le:Rni} directly yields the result.
\end{proof}

\section*{Acknowledgements}
The author was partially supported by the German Academic
Exchange Service (DAAD) via grant 57468851 and by DFG
priority program SPP 2265 \textit{Random Geometric Systems}.
\bibliographystyle{alpha} 
\bibliography{IsolatedVertices}

\end{document}